\documentclass[a4paper,11pt]{article}

\usepackage{bm}
\usepackage{amsmath}
\usepackage[latin1]{inputenc}
\usepackage{graphics}
\usepackage{verbatim}
\usepackage{listings}
\usepackage{dsfont}
\usepackage{amsmath}
\usepackage{amsfonts}
\usepackage{amssymb}
\usepackage{graphicx}
\usepackage{url}
\usepackage[latin1]{inputenc}
\usepackage{color}
\usepackage{textcomp}

\newtheorem{set2}{Satz}[section]
\newtheorem{theorem}[set2]{Theorem}

\newtheorem{corollary}[set2]{Corollary}

\newtheorem{definition}[set2]{Definition}

\newtheorem{lemma}[set2]{Lemma}
\newtheorem{notation[set2]}{Notation}

\newtheorem{remark}[set2]{Remark}

\newcommand{\ep}{\hfill{$\square$}}
\newenvironment{proof}[1][Proof]{\textit{#1.} }{\
\\}

\def\XXint#1#2#3{{\setbox0=\hbox{$#1{#2#3}{\int}$}
\vcenter{\hbox{$#2#3$}}\kern-.5\wd0}}

\newcommand{\R}{\mathbb{R}}
\newcommand{\N}{\mathbb{N}}
\newcommand{\C}{\mathcal}
\newcommand{\ol}{\overline}
\newcommand{\dx}{\,\mathrm dx}

\newcommand{\ds}{\,\mathrm ds}

\newcommand{\dxt}{\,\mathrm dx\,\mathrm dt}
\newcommand{\dxs}{\,\mathrm dx\,\mathrm ds}

\newcommand{\weaklim}{\rightharpoonup}

\newcommand{\interior}{\mathrm{int}}

\newcommand{\essinf}{\mathrm{ess\,inf}}
\newcommand{\compact}{\subset\!\subset}

\newcommand{\argmin}{\mathrm{arg\,min}}
\newcommand{\e}{\epsilon}

\usepackage{geometry}
\allowdisplaybreaks

\date{\today}

\title{A degenerating Cahn-Hilliard system coupled with complete damage processes}
\author{\nofnmark{}Christian Heinemann, Christiane Kraus
\footnote{Weierstra\ss-Institut\\
Mohrenstr. 39\\ 10117 Berlin \\ Germany\\E-Mail: \email{christian.heinemann@wias-berlin.de}\\
\hspace{3.6em}\email{christiane.kraus@wias-berlin.de}}}

\begin{document}
\begin{center}
\Large
A degenerating Cahn-Hilliard system coupled with complete damage processes
\end{center}
\begin{center}
	Christian Heinemann\footnote{Weierstrass Institute for Applied Analysis and Stochastics (WIAS), Mohrenstr. 39, 10117 Berlin, Germany.\\
		This project is supported by the DFG Research Center  
``Mathematics for Key Technologies''  Matheon in Berlin.\\
E-mail: \texttt{christian.heinemann@wias-berlin.de} and
\texttt{christiane.kraus@wias-berlin.de}},
	Christiane Kraus$^1$
\end{center}
\begin{center}
	December 30, 2012
\end{center}

\vspace{2mm}
\noindent

\begin{abstract}
	In this work, we analytically investigate a degenerating PDE system
	for phase separation and complete damage processes
	considered on a nonsmooth time-dependent domain
	with mixed boundary conditions. The evolution of the system is
	described by a \textit{degenerating} Cahn-Hilliard equation for the concentration, a doubly nonlinear differential inclusion for the damage
	variable and a \textit{degenerating} quasi-static balance equation for the displacement
	field. All these equations are highly nonlinearly coupled.
	Because of the doubly degenerating character of the system, the doubly nonlinear differential inclusion and the nonsmooth domain,
	the structure of the model is very complex from an analytical point of view.
	
	A novel approach is introduced for proving existence of weak solutions for such degenerating coupled system.
	To this end, we first establish a suitable notion of weak solutions, which consists of weak formulations of the
	diffusion and the momentum balance equation, a variational inequality for the damage process and a total energy
	inequality.
	To show existence of weak solutions, several new ideas come into play.
	Various results on shrinking sets and its corresponding local Sobolev spaces are used. It turns out that, for instance,
	on open sets which shrink in time a quite satisfying analysis in Sobolev spaces is possible.
	The presented analysis can handle highly nonsmooth regions where complete damage takes place.
	To mention only one difficulty, infinitely many completely damaged regions which are not connected with
	the Dirichlet boundary may occur in arbitrary small time intervals.

\end{abstract}

\rm

\section{Problem description}
\label{section:intro}
	Phase separation and damage processes occur in many fields, including material sciences, biology and chemical reactions.
	In particular, for the manufacturing and lifetime prediction of micro-electronic devices it is of great importance
	to understand the mechanisms and the interplay between phase-separation and damage processes in solder alloys.
	As soon as elastic alloys are quenched sufficiently, spinodal decomposition leads to a fine-grained structure of
	different chemical mixtures on a short time-scale (see \cite{DM01} for numerical simulations and experimental observations).
	The long-term evolution is determined by a chemical diffusion process which tends to minimize the bulk and the surface energy of the chemical substances.
	J.W. Cahn and J.E. Hilliard developed a phenomenological model for the kinetics of phase-separation in a thermodynamically consistent framework
	known as the Cahn-Hilliard equation \cite{CH58}, for which an extensive mathematical literature exists.
	An overview of modeling and analytical aspects of the Cahn-Hilliard equation can be found in \cite{El89}.
	The recent literature is mainly focused on coupled systems.
	For instance, physical observations and numerical simulations reveal that mechanical stresses influence the developing of shapes of the chemical phases.
	A coupling between Cahn-Hilliard systems and elastic deformations have
	been analytically studied in \cite{ GarckeHabil, Bonetti02, CMP00, Garcke052, Garcke05, Pawlow, Pawlow08}. For numerical results and simulations we refer to
	\cite{Wei02,Mer05,BB99,GRW01,BM2010}.
	Phase separation of the chemical components may also lead to critical stresses at phase boundaries due to swelling which result in cracks and formation of
	voids and are of particular interest to understand the aging process in solder materials, cf. \cite{HCW91,Ubachs07,Gee07, FK09}.
	A fully coupled system consisting of the Cahn-Hilliard equation, an
        elliptic equation for the displacement field and a differential
        inclusion for the damage variable has been recently investigated in
        \cite{WIAS1520, WIAS1569}.
	However, in \cite{WIAS1520, WIAS1569} and in the most mathematical damage literature \cite{BS04, Gia05, Mielke06, MT10, KRZ11}, it is usually assumed that damage cannot completely
	disintegrate the material (i.e. \textit{incomplete damage}). Dropping
        this assumption gives rise to many mathematical challenges. Therefore,
        global-in-time existence results for complete damage models are rare. 
	Modeling and existence of weak solutions for \textit{purely} mechanical complete damage systems with quasi-static force balances are studied in \cite{BMR09,Mie11,WIAS1722} and
	with visco-elasticity in \cite{Mielke10,RR12}.
	
	The main goal in the present work is to prove existence of weak
	solutions of a system \textit{coupling} a differential inclusion describing damage processes with an 
	elastic Cahn-Hilliard system in a small strain setting as in \cite{WIAS1520} but allowing for \textit{complete damage}
	and \textit{degenerating mobilities} with respect to the damage variable
	(see Theorem \ref{theorem:globalInTimeExistence}).
	Note that the mobility is still assumed to be independent of the concentration variable (see \cite{EG96} for an analysis of Cahn-Hilliard equations with degenerating
	and concentration dependent mobilities).
	The elasticity is considered to be linear and the system is assumed to
	be in quasi-static mechanical equilibrium since diffusion processes take place on a much slower time scale.
	The main modeling idea in \cite{WIAS1722} has been
	to formulate such degenerating system on a time-dependent domain which consists of the not completely damaged regions that are connected to the
	Dirichlet boundary.
	In this context, the concept of maximal admissible subsets is
	introduced to specify the domain of interest.

	In the following,
	we fix a bounded $\C C^2$-domain $\Omega\subseteq\R^n$ and a Dirichlet boundary part $D\subseteq\partial\Omega$ with $\C H^{n-1}(D)>0$.
	
	A relatively open subset $G$ of $\ol\Omega$ is called
        \textit{admissible} with respect to an
        $\C H^{n-1}$-measurable part $D$ of the boundary $\partial\Omega$ if every
	path-connected component $P_G$ of $G$ satisfies $\C H^{n-1}(P_G\cap D)>0$,
	where $\C H^{n-1}$ denotes the $(n-1)$-dimensional Hausdorff measure.
	The \textit{maximal admissible subset} of $G$ is denoted by $\frak A_D(G)$.
	With the notion of maximal admissible subsets, we can formulate our evolutionary system with a time-dependent domain.
	
	\textbf{Degenerating PDE system on a time-dependent domain.}
	
	\textit{Find
	functions
	}
	\begin{align*}
		c\in\mathcal C^2(F;\R),\;\quad u\in\mathcal C_x^2(F;\R^n),\;\quad z\in \C C^2(\ol{\Omega_T};\R),\;\quad \mu\in\mathcal C_x^2(F;\R)
	\end{align*}
	with the time-dependent domain $F=\{(x,t)\in\ol{\Omega_T}\,|\,x\in F(t)\}$ and $F(t)=\frak A_D\big(\{z(t)>0\}\big)$
	\textit{such that the PDE system}
	\begin{align*}
	\begin{split}
		&0=\mathrm{div}(W_{,e}(c,\epsilon(u),z)),\\
		&c_t=\mathrm{div}(m(z)\nabla\mu),\\
		&\mu= -\Delta c+\Psi_{,c}(c)+W_{,c}(c,\epsilon(u),z),\\
		&z_t+\xi-\mathrm{div}(|\nabla z|^{p-2}\nabla z)+W_{,z}(c,\epsilon(u),z)+f'(z)=0,\\
		&\xi\in\partial I_{(-\infty,0]}(z_t)
	\end{split}
	\end{align*}
	\textit{is satisfied pointwise in $\interior(F)$ with the initial-boundary conditions}\vspace*{0.5em}\\
	\begin{tabular}{p{0em}ll}
		&\renewcommand {\baselinestretch} {1.3} \normalsize
		\begin{minipage}{21em}
			{
			$c(0)=c^0$, $z(0)=z^0$\\
			$u(t)=b(t)$\\
			$W_{,e}(c(t),\epsilon(u(t)),z(t))\cdot\nu=0$\\
			$z(t)=0$\\
			$\nabla z(t)\cdot\nu=0$\\
			$\nabla c(t)\cdot\nu=0$\\
			$m(z(t))\nabla\mu(t)\cdot\nu=0$\\
			}
		\end{minipage}
		&\renewcommand {\baselinestretch} {1.3} \normalsize
		\begin{minipage}{18em}
			{
			\textit{in} $F(0)$,\\
			\textit{on} $\Gamma_1(t):=F(t)\cap D$,\\
			\textit{on} $\Gamma_2(t):=F(t)\cap (\partial\Omega\setminus D)$,\\
			\textit{on} $\Gamma_3(t):=\partial F(t)\setminus F(t)$,\\
			\textit{on} $\Gamma_1(t)\cup\Gamma_2(t)$,\\
			\textit{on} $\Gamma_1(t)\cup\Gamma_2(t)$,\\
			\textit{on} $\Gamma_1(t)\cup\Gamma_2(t)$.\\
			}
		\end{minipage}
	\end{tabular}
	
	Here, $\C C_x^2(F;\R^N)$ denotes the space of two times continuously
        differentiable functions with respect to the spatial variable $x$ on the set $F$.
        
	The solution of the PDE system can physically be interpreted as follows:
	$c$ denotes the chemical concentration difference of a two-component alloy, $\epsilon(u):=\frac12(\nabla u+(\nabla u)^T)$ the linearized strain tensor of the displacement $u$,
	$z$ the damage profile describing the degree of damage (i.e. $z=1$ undamaged and $z=0$ completely damaged material point) and $\mu$ the chemical potential.
	Moreover, $W$ denotes the elastic energy density, $\Psi$ the chemical energy density, $f$ a damage dependent potential,
	$m$ the mobility depending on the damage variable $z$ and $b$ the time-dependent Dirichlet boundary data for $D$.
	\textit{Capillarity effects} are modeled by concentration gradients in the free energy which provokes the chemical potential to depend on the Laplacian $\Delta c$
	(see \cite{Gu89}).
	
	We assume the following product structure for the elastic energy density:
	\begin{align}
	\label{eqn:elasticEnergyStr}
		W(c,e,z)=g(z)\varphi(c,e)
	\end{align}
	with a non-negative monotonically increasing function $g\in\C C^1([0,1];\R^+)$ such that the complete damage condition $g(0)=0$ is fulfilled.
	The second function $\varphi\in\C C^1(\R\times\R_\mathrm{sym}^{n\times n};\R^+)$ should have the following polynomial form
	\begin{align}
	\label{eqn:elasticEnergy}
		&\varphi(c,e)=\varphi^1e:e+\varphi^2(c):e+\varphi^3(c)
	\end{align}
	for coefficients $\varphi^1\in \mathcal L(\mathbb
        R_\mathrm{sym}^{n\times n})$ with $\varphi^1$ positive definite, 
	$\varphi^2\in \mathcal C^1(\mathbb R;\mathbb R_\mathrm{sym}^{n\times n})$
	and $\varphi^3\in \mathcal C^1(\mathbb R)$.
	Note that homogeneous elastic energy densities of the type
	$$W(c,e,z)=\frac 12z\mathbb C(e-e^\star(c)):(e-e^\star(c))$$
	with eigenstrain $e^\star$ depending on the concentration and stiffness tensor $\mathbb C$ are covered in this approach.
	We remark that this approach allows to incorporate swelling phenomena of the chemical species which, in turn,
	can lead to critical stresses for the initiation of damage.
	The mobility $m\in \C C([0,1];\R^+)$, on the other hand, should satisfy the following condition for degeneracy, i.e. 
	\begin{align}
		\label{eqn:assumptionMobility}
			m(z)=0\text{ if and only if } z=0.
	\end{align}
	\textbf{Plan of the paper}
	
	A weak formulation of the degenerating system is given in Section \ref{section:weakSolution} while the proof of existence is carried out in Section \ref{section:degLimit}
	and Section \ref{section:existence}.
	In the first step of the proof, a degenerated limit of the corresponding incomplete damage system coupled with elastic Cahn-Hilliard equations is performed (see 
	Section \ref{section:degLimit}).
	Due to the additional coupling the passage to the limit becomes quite involved and a
	\textit{conical Poincar\'e inequality} is used to control the chemical potential in a local sense.
	
	In the next step, we take material exclusions into account (see Section \ref{section:existence}). They occur when not completely damaged components become disconnected to the Dirichlet boundary.
	The exclusions lead to jumps in the energy inequality. By a concatenation property and by Zorn's lemma, we are able to prove
	the main theorem in this paper: a global-in-time existence result (see Theorem \ref{theorem:globalInTimeExistence}).
	In addition, several $\Gamma$-results and various properties on shrinking sets and their corresponding local Sobolev spaces are used.
	
\section{Assumptions and notion of weak solutions}
\label{section:weakSolution}
	As before, $\Omega\subseteq\R^n$ denotes a bounded $\C C^2$-domain and $D\subseteq\partial\Omega$ the Dirichlet boundary with $\C H^{n-1}(D)>0$.
	Furthermore, $T>0$ denotes the maximal time of interest and $p>n$ a fixed exponent.

	The weak formulation of the PDE system presented in this work will be based on an energetic approach and uses the associated free energy.
	Let $G\subseteq\ol\Omega$ be a relatively open subset. Then, the free energy $\C E$ on $G$ is given by 
	$$
		\C E_G(c,e,z):=\int_{G} \Big( \frac 1p|\nabla z|^p+\frac
                12|\nabla c|^2+\Psi(c)+W(c,e,z)+f(z)\Big) \dx
	$$
	for $c\in H^1(G)$, $e\in L^2(G;\R_\mathrm{sym}^{n\times n})$ and $z\in W^{1,p}(G)$.
	
	We suppose the structural assumptions \eqref{eqn:elasticEnergyStr} and \eqref{eqn:elasticEnergy}.
	Here, $\Psi\in C^1(\R)$ denotes the chemical energy density, $f\in C^1([0,1])$ a damage dependent potential
	and $g\in C^1([0,1];\R^+)$ with $g(0)=0$ a function modeling the influence of the damage on the material stiffness.
	We assume the following growth conditions:
	\begin{subequations}
	\begin{align}
	\label{eqn:assumptionPhi2}
		|\varphi^2(c)|,|\varphi^{2}_{,c}(c)|&\leq C(1+|c|),\\
	\label{eqn:assumptionPhi3}
		|\varphi^3(c)|,|\varphi^{3}_{,c}(c)|&\leq C(1+|c|^2),\\
	\label{eqn:assumptionPsi}
		|\Psi_{,c}(c)|&\leq C(1+|c|^{2^\star/2}),\\
	\label{eqn:assumptiong}
		\eta&\leq g'(z).
	\end{align}
	\end{subequations}
	The constants $\eta,C>0$ are independent of $c$ and $z$, and
	$2^\star$ denotes the Sobolev critical exponent.
	In the case $n=2$, $\Psi_{,c}$ has to satisfy an $r$-growth condition for a fixed arbitrary $r>0$ whereas
	we have no restrictions on $\Psi_{,c}$ in the one-dimensional case.

	The functions $m\in \C C([0,1];\R^+)$ will denote the mobility of the Cahn-Hilliard diffusion process
	and is assumed to satisfy \eqref{eqn:assumptionMobility}.

	Before we state the precise notion of weak solutions, we introduce some notation.
        
	The set $\{f>0\}$ for a given function $f\in W^{1,p}(\Omega)$ has to be read as $\{x\in\ol\Omega\,|\,f(x)>0\}$ by employing the embedding
	$W^{1,p}(\Omega)\hookrightarrow\C C(\ol\Omega)$ (because of
        $p>n$). The symbol $H^1_{S}(G)$ denotes the subspace of $H^1(G)$ with zero
        trace on $S$, where  $G \subset \R^n$ is open and $S$ is an $\C
        H^{n-1}$-measurable subset of $\partial \Omega$.
        
        Let $I \subseteq \R$ be an open interval. The subspace
        $SBV(I;X)\subseteq BV(I;X)$ 
        of special functions of bounded variation is defined as the space of functions
	$f\in BV(I;X)$ where the decomposition
	\begin{align*}
		\mathrm df = f'\C L^1 + (f^+-f^-)\C H^0\lfloor J_f
	\end{align*}
	for an $f'\in L^1(I;X)$ exists.
	The function $f'$ is called the absolutely continuous part of the differential measure
	and we also write $\partial_t^\mathrm{a}f$ and $J_f$ denotes the jump set.
	If, additionally, $\partial_t^\mathrm{a}f\in L^p(I;X)$, $p\geq 1$, we write $f\in SBV^p(I;X)$.
	
	We say that a relatively open subset $F\subseteq\ol{\Omega_T}$ is shrinking iff
	$F(t)\subseteq F(s)$ for all $0\leq s\leq t\leq T$, where $F(t)$ denotes the $t$-cut of $F$.
	For the analysis of our proposed system, we introduce local Sobolev functions.
	The space-time local Sobolev space $L_t^2
	H_{x,\mathrm{loc}}^1(F;\R^N)$ for a shrinking set $F$ is given by
	\begin{align*}
	  L_t^2H^q_{x,\mathrm{loc}}(F;\R^N)
	  :=\Big\{& \! v:\!F\rightarrow\R^N\;\Big|\;
	  \forall t\in(0,T],\,\forall U\compact F(t)\,\text{open}:\;
	    v|_{U\times (0,t)}\in L^{2}(0,t;H^{q}(U;\R^N))\Big\}.
	\end{align*}
	
	In the following, a weak formulation of the system above combining the ideas in \cite{WIAS1520} and \cite{WIAS1722} is given.
	
	\begin{definition}[Weak solution of the coupled PDE system] A quadruple $(c,u,z,\mu)$ is called a weak solution
		with the initial-boundary data $(c^0,z^0,b)$ if
		\label{def:weakSolution}
		\begin{enumerate}
			\renewcommand{\labelenumi}{(\roman{enumi})}
			\item
				\textit{Spaces:}
				\begin{align*}
				\begin{aligned}
					&c\in L^\infty(0,T;H^1(\Omega))\cap H^1(0,T;(H^1(\Omega))^\star),\;&&c(0)=c^0,\\
					&u\in L_t^2 H_{x,\mathrm{loc}}^1(F;\R^n),\;&&u=b\text{ on }D_T\cap F,\\
					&z\in L^\infty(0,T;W^{1,p}(\Omega))\cap SBV^2(0,T;L^2(\Omega)),\;&& z(0)=z^0,\;z^+(t)=z^-(t)\mathds 1_{F(t)}\\
					&\mu\in L_t^2 H_{x,\mathrm{loc}}^1(F)
				\end{aligned}
				\end{align*}
				with $e:=\epsilon(u)\in L^2(F;\R_\mathrm{sym}^{n\times n})$ where $F:=\frak A_D(\{z^->0\})\subseteq\ol{\Omega_T}$ is a shrinking set\\
				($z^+(t):=\lim_{s\downarrow t}z(s)$ and $z^-(t):=\lim_{s\uparrow t}z(s)$ in $L^2(\Omega)$).
			\item
				\textit{Quasi-static mechanical equilibrium:}
				\begin{align}
				\label{eqn:forceBalanceWeak}
					0=\int_{F(t)}W_{,e}(c(t),e(t),z(t)):\epsilon(\zeta)\dx
				\end{align}
				for a.e. $t\in(0,T)$ and for all $\zeta\in H_{D}^1(\Omega;\R^n)$.
			\item
				\textit{Diffusion:}
				\begin{align}
				\label{eqn:diffusion}
					\int_{\Omega_T}\partial_t\zeta(c-c^0)\dxt=\int_{F}m(z)\nabla\mu\cdot\nabla\zeta\dxt
				\end{align}
				for all $\zeta\in L^2(0,T;H^1(\Omega))\cap H^1(0,T;L^2(\Omega))$ with $\zeta(T)=0$ and
				\begin{align}
				\label{eqn:chemicalPotential}
					\int_{F}\mu\zeta\dxt=\int_{F}\Big(\nabla c\cdot\nabla\zeta+\Psi_{,c}(c)\zeta+W_{,c}(c,e,z)\zeta\Big)\dxt
				\end{align}
				for all $\zeta\in L^2(0,T;H^1(\Omega))$ with $\mathrm{supp}(\zeta)\subseteq F$.
			\item
				\textit{Damage variational inequality: }
				\begin{align}
				\label{eqn:VI}
					&0\leq\int_{F(t)}\Big(|\nabla z(t)|^{p-2}\nabla z(t)\cdot\nabla\zeta+\big(W_{,z}(c(t),e(t),z(t))+f'(z(t))+\partial_t^\mathrm{a} z(t)\big)\zeta\Big)\dx\\
					&0\leq z(t),\notag\\
					&0\geq\partial_t^\mathrm{a} z(t)\notag
				\end{align}
				for a.e. $t\in(0,T)$ and for all $\zeta\in W^{1,p}(\Omega)$ with $\zeta\leq 0$.
			\item
				\textit{Energy inequality: }
				\begin{align}
				\label{eqn:EI}
					\C E(t)+\C J(0,t)+\C D(0,t)\leq \C E^+(0)+\C W_\mathrm{ext}(0,t),
				\end{align}
				where the terms in the energy inequality are as follows:
				\begin{align*}
					&\textit{\textbullet\;Energy: }&&\C E(t):=\C E_{F(t)}(c(t),e(t),z(t)),\\
					&\textit{\textbullet\;Energy jump term: }&&\C J(0,t):=\sum_{s\in J_{z}\cap (0,t]}\left(\C E^-(s)-\C E^+(s)\right),\\
						&&&\C E^-(t):=\lim_{s\to t^-}\Big(\mathop\essinf_{\tau\in(s,t)}\C E(\tau)\Big),\\
						&&&\C E^+(t):=\frak e_t^+\\
						&&&\text{with }0\leq \frak e_t^+\leq \inf_{\zeta\in H_{D\cap F(t)}^1(F(t);\R^n)}\C E_{F(t)}(c(t),\e(b(t)+\zeta),z(t)),\\
					&\textit{\textbullet\;Dissipated energy: }&&\C D(0,t):=\int_0^t\int_{F(s)}\Big(m(z)|\nabla\mu|^2+|\partial_t^\mathrm{a} z|^2\Big)\dxt,\\
					&\textit{\textbullet\;External work: }&&\C W_\mathrm{ext}(0,t):=\int_0^t\int_{F(s)}W_{,e}(c,\e(u),z):\e(\partial_t b)\dxt.
				\end{align*}
		\end{enumerate}
	\end{definition}
	\begin{remark}
	\label{remark:weakSolution}
	Under additional regularity assumptions, a weak solution
	satisfies
	\begin{align*}
		&\lim_{\tau\to s^-}\C E_{F(\tau)}(c(\tau),e(\tau),z(\tau))
			=\lim_{\tau\rightarrow s^-}\mathop\mathrm{ess\,inf}_{\vartheta\in (\tau,s)}\C E_{F(\vartheta)}(c(\vartheta),e(\vartheta),z(\vartheta)),\\
		&\lim_{\tau\to s^+}\C E_{F(\tau)}(c(\tau),e(\tau),z(\tau))=\frak e_s^+
	\end{align*}
	and reduces to the pointwise classical formulation presented in Section \ref{section:intro}
	(cf. \cite[Theorem 3.7]{WIAS1722}).
	\end{remark}

	As mentioned before, the main aim of this paper is to prove global-in-time existence of weak solutions.
	Due to the analytical quite involved structure, we introduce the notion
	of weak solutions with a given fineness constant $\eta>0$.
	\begin{definition}[Weak solution with fineness ${\bm \eta}$]
	\label{def:approxWeakSolution}
		A tuple $(c,e,u,z,\mu)$ together with a shrinking set $F\subseteq\ol{\Omega_T}$ is called a weak solution with fineness $\eta>0$
		if
		\begin{align*}
		\begin{aligned}
			&c\in L^\infty(0,T;H^1(\Omega))\cap H^1(0,T;(H^1(\Omega))^\star),\;&&c(0)=c^0,\\
			&u\in L_t^2 H_{x,\mathrm{loc}}^1(\frak A_D(F);\R^n),\;&&u=b\text{ on }D_T\cap \frak A_D(F),\\
			&z\in L^\infty(0,T;W^{1,p}(\Omega))\cap SBV^2(0,T;L^2(\Omega)),\;&&z(0)=z^0,\;z^+(t)=z^-(t)\mathds 1_{F(t)}\\
			&\mu\in L_t^2 H_{x,\mathrm{loc}}^1(F)
		\end{aligned}
		\end{align*}
		such that
		$e=\epsilon(u)$ in $\frak A_D(F)$ and 
		the properties (ii)-(v) of Definition \ref{def:weakSolution} and
		\begin{align*}
			&\forall t\in[0,T]:&&\frak A_D(\{z^-(t)>0\})\subseteq F(t)\text{ and }\C L^n\big(F(t)\setminus \frak A_D(\{z^-(t)>0\}))\big)<\eta,\\
			&\forall t\in[0,T]\setminus \bigcup_{t\in C_{z^\star}}[t,t+\eta)&&\frak A_D(\{z^-(t)>0\})=F(t).
		\end{align*}
		are satisfied.
	\end{definition}
	Here, $C_{z^\star}$ denotes the set of cluster points from the right of the jump set $J_{z^{\star}}$
	of the function $z^{\star}\in SBV^2(0,T;L^2(\Omega))$ given by $z^{\star}(t):=z(t)\mathds 1_{\frak A_D(\{z^-(t)>0\})}$, i.e.,
	$z^\star$ is the damage profile of $z$ restricted to the time-dependent domain $\frak A_D(\{z^->0\})$
	($\mathds 1_A:X\rightarrow\{0,1\}$ denotes the characteristic function of a set $A\subseteq X$).
	
	\begin{remark}
		\begin{itemize}
			\item[(i)]
				If a weak solution $(c,e,u,z,\mu)$ on $F$ with fineness $\eta$ according to Definition \ref{def:approxWeakSolution}
				satisfies $C_{z^{\star}}=\emptyset$ then $(c,u,z,\mu)$ is a weak solution
				according to Definition \ref{def:weakSolution}.
			\item[(ii)] 
				If the initial value for the damage profile contains no completely damaged parts on $\ol\Omega$, we are also able
				to establish maximal local-in-time existence of weak solutions. More precisely,
				there exist a maximal value $\widehat T>0$ with $\widehat T\leq T$
				and functions $c$, $u$, $z$ and $\mu$ defined on the time interval $[0,\widehat T]$
				such that $(c, u,z, \mu)$ is a weak solution according to Definition \ref{def:weakSolution}.
				Therefore, if $\widehat T<T$, $(c,u,z,\mu)$ cannot be extended
				to a weak solution on $[0,\widehat T+\varepsilon]$.
		\end{itemize}
	\end{remark}
	The remaining part of this paper is devoted to establish a global existence result in the sense of Definition \ref{def:approxWeakSolution}
	(see Theorem \ref{theorem:globalInTimeExistence}).

\section{Degenerate limit of the regularized system}
\label{section:degLimit}
	In this section, we will start with a corresponding incomplete damage
        model coupled to an elastic Cahn-Hilliard system and then perform a limit procedure.
	For each $\varepsilon>0$, we define the regularized free energy $\C E_\varepsilon$ as
	\begin{align*}
		\C E_\varepsilon(c,e,z):=\int_{\Omega}\Big(\frac 1p|\nabla z|^p+\frac 12|\nabla c|^2+\Psi(c)+W^\varepsilon(c,e,z)+f(z)\Big)\dx
	\end{align*}
	for functions $c\in H^1(\Omega)$, $e\in
        L^2(\Omega;\R_\mathrm{sym}^{n\times n})$, $z\in W^{1,p}(\Omega)$. 
	The regularized elastic energy density and mobility are given by
	\begin{align*}
	\begin{split}
		W^\varepsilon(c,e,z)&:=(g(z)+\varepsilon)\varphi(c,e),\\
		m^\varepsilon(z)&:=m(z)+\varepsilon.
	\end{split}
	\end{align*}
	
	A modification of the proof of Theorem 4.6 in \cite{WIAS1520} yields the following result.
	\begin{theorem}[$\varepsilon$-regularized coupled PDE problem]
	\label{theorem:regProblem}
		Let $\varepsilon>0$.
		For given initial-boun\-dary data $c_\varepsilon^0\in H^1(\Omega)$, $z_\varepsilon^0\in W^{1,p}(\Omega)$
		and $b_\varepsilon\in W^{1,1}(0,T;W^{1,\infty}(\Omega;\R^n))$
		there exists a quadruple
		$q_\varepsilon=(c_\varepsilon,u_\varepsilon,z_\varepsilon, \mu_\varepsilon)$ such that
		\begin{enumerate}
			\renewcommand{\labelenumi}{(\roman{enumi})}
			\item
				\textit{Spaces:}
				\begin{align*}
				\begin{aligned}
					&c_\varepsilon\in L^\infty(0,T;H^1(\Omega))\cap H^1(0,T;(H^1(\Omega))^\star),\;&&c_\varepsilon(0)=c_\varepsilon^0,\\
					&u_\varepsilon\in L^\infty(0,T;H^1(\Omega;\mathbb R^n)),\;&&u_\varepsilon=b_\varepsilon\text{ on }D_T,\\
					&z_\varepsilon\in L^\infty(0,T;W^{1,p}(\Omega))\cap H^1(0,T;L^2(\Omega)),\;&& z_\varepsilon(0)=z_\varepsilon^0\\
					&\mu_\varepsilon\in L^2(0,T;H^1(\Omega))
				\end{aligned}
				\end{align*}
			\item
			\textit{Quasi-static mechanical equilibrium:}
				\begin{equation}
				\label{eqn:ID1}
					\int_{\Omega} W_{,e}^\varepsilon(c_\varepsilon(t),\epsilon(u_\varepsilon(t)),z_\varepsilon(t)):\epsilon(\zeta)\,\mathrm dx=0
				\end{equation}
				for a.e. $t\in(0,T)$ and for all $\zeta\in H_D^{1}(\Omega;\mathbb R^n)$.
			\item 
			\textit{Diffusion:}
				\begin{equation}
				\label{eqn:ID2}
					\int_{\Omega_T}(c_\varepsilon-c_\varepsilon^0)\partial_t\zeta\dxt
						=\int_{\Omega_T}m^\varepsilon(z_\varepsilon)\nabla \mu_\varepsilon\cdot\nabla\zeta\dxt
				\end{equation}
				for all $\zeta\in L^2(0,T;H^1(\Omega))$ with $\partial_t\zeta\in L^2(\Omega_T)$ and $\zeta(T)=0$ and
				\begin{align}
						\int_{\Omega} \mu_\varepsilon(t)\zeta\dx
							={}&\int_{\Omega}\Big(\nabla c_\varepsilon(t)\cdot\nabla\zeta+\Psi_\mathrm{,c}(c_\varepsilon(t))\zeta
					\label{eqn:ID3}
						+W_\mathrm{,c}^\varepsilon(c_\varepsilon(t),\epsilon(u_\varepsilon(t)),z_\varepsilon(t))\zeta\Big)\dx
				\end{align}
				for a.e. $t\in(0,T)$ and for all $\zeta\in H^1(\Omega)$.
			\item
			\textit{Damage variational inequality: }
				\begin{align}
				\label{eqn:ID4}
					&0\leq\int_{\Omega}\Big(|\nabla z_\varepsilon(t)|^{p-2}\nabla z_\varepsilon(t)\cdot\nabla\zeta
						+\big(W_{,z}^\varepsilon(c_\varepsilon(t),\epsilon(u_\varepsilon(t)),z_\varepsilon(t))+f'(z_\varepsilon(t))
						+\partial_t z_\varepsilon(t)+r_\varepsilon(t)\big)\zeta\Big)\dx\\
					&0\leq z_\varepsilon(t),\notag\\
					&0\geq\partial_t z_\varepsilon(t)\notag
				\end{align}
				for a.e. $t\in(0,T)$ and for all $\zeta\in W^{1,p}(\Omega)$ with $\zeta\leq 0$,
				where $r_\varepsilon\in L^1(\Omega_T)$ is given by
				\begin{align}
				\label{eqn:rstructure}
					r_\varepsilon=-\chi_\varepsilon \big(W_{,z}^\varepsilon(c_\varepsilon,\epsilon(u_\varepsilon),z_\varepsilon)+f'(z)\big)^+
				\end{align}
				with $\chi_\varepsilon\in L^\infty(\Omega)$ fulfilling $\chi_\varepsilon=0$ on $\{z_\varepsilon>0\}$ and
				$0\leq\chi_\varepsilon\leq 1$ on $\{z_\varepsilon=0\}$.
		\item
			\textit{Energy inequality: }
				\begin{align}
					& \C E_\varepsilon(c_\varepsilon(t),\epsilon(u_\varepsilon(t)),z_\varepsilon(t))
						+\int_{\Omega_t}\Big(m^\varepsilon(z_\varepsilon)|\nabla \mu_\varepsilon|^2+|\partial_t z_\varepsilon|^2\Big)\dxs\notag\\
				\label{eqn:ID5}
					&\qquad\qquad\leq \C E_\varepsilon(c_\varepsilon^0,\epsilon(u_\varepsilon^0),z_\varepsilon^0)
						+\int_{\Omega_t} W_{,e}(c_\varepsilon,\epsilon(u_\varepsilon),z_\varepsilon):
						\epsilon(\partial_t b_\varepsilon)\dxs
				\end{align}
				holds for a.e. $t\in(0,T)$ where $u_\varepsilon^0$ minimizes $ \C E_\varepsilon(c^0,\epsilon(\cdot),z_\varepsilon^0)$ in $H^1(\Omega;\R^n)$
				with Dirichlet data $b_\varepsilon^0:=b_\varepsilon(0)$ on $D$.
			\end{enumerate}
	\end{theorem}
	\begin{proof}
		The existence theorem presented in \cite{WIAS1520} can be adapted
		to our situation by considering the viscous semi-implicit time-discretized system (in a classical notation; we omit the $\varepsilon$-dependence
		in the notation for the discrete solution at the moment):
		\begin{align*}
			0&=\mathrm{div}\big(W_{,e}^\varepsilon(c^k,\epsilon(u^k),z^k)\big)+\delta\,\mathrm{div}(|\nabla u^k|^2\nabla u^k),\\
			\frac{c^k-c^{k-1}}{\tau}&=\mathrm{div}(m^\varepsilon(z^{k-1})\nabla\mu^k),\\
			\mu^k&= -\Delta c^k+\Psi_{,c}(c^k)+W_{,c}^\varepsilon(c^k,\epsilon(u^k),z^k)+\delta\frac{c^k-c^{k-1}}{\tau},\\
			\frac{z^k-z^{k-1}}{\tau}+\xi+\zeta&=\mathrm{div}(|\nabla z^k|^{p-2}\nabla z^k)+W_{,z}^\varepsilon(c^k,\epsilon(u^k),z^k)+f'(z^k),
		\end{align*}
		with the sub-gradients $\xi\in \partial I_{(-\infty,0]}((z^k-z^{k-1})/\tau)$, $\zeta\in \partial I_{[0,\infty)}(z^k)$
		and the discretization fineness $\tau=T/M$ for $M\in\N$.
		The discrete equations can be obtained recursively starting from $(c^0,u^0,z^0)$
		with $u^0:=\argmin_{u\in H^1(\Omega;\R^n),\,u|_D=b^0|_D}\C E_\varepsilon(c^0,u,z^0)$
		by considering the Euler-Lagrange equations of the functional
		\begin{align*}
		\begin{split}
			\mathbb E^k(c,u,z):={}&\C E_\varepsilon(c,u,z)+\int_\Omega \frac \delta4|\nabla u|^4\dx\\
			&+\frac \tau2\left(\left\|\frac{z-z^{k-1}}{\tau}\right\|_{L^2(\Omega)}^2
				+\left\|\frac{c-c^{k-1}}{\tau}\right\|_{X(z^{k-1})}^2+\delta\left\|\frac{c-c^{k-1}}{\tau}\right\|_{L^2(\Omega)}^2\right)
		\end{split}
		\end{align*}
		defined on the subspace of $H^1(\Omega) \times W^{1,4}(\Omega;\mathbb R^n)\times W^{1,p}(\Omega)$ with the conditions
		$u|_D=b(k\tau)|_D$, $\int_\Omega(c-c^0)\,\mathrm dx=0$ and
		$0\leq z\leq z^{k-1}$ a.e. in $\Omega$.
		The weighted scalar product $\langle \cdot,\cdot\rangle_{X(z^{k-1})}$ is given by
		\begin{align*}
			\langle u,v\rangle_{X(z^{k-1})}:=\left\langle m^\varepsilon(z^{k-1})\nabla A^{-1} u,\nabla A^{-1}v\right\rangle_{L^2(\Omega)}
		\end{align*}
		with the operator $A:V_0\rightarrow \tilde{V}_0$,
		$Au:=\left\langle m^\varepsilon(z^{k-1})\nabla u,\nabla\cdot\right\rangle_{L^2(\Omega)}$ and the spaces
		\begin{equation*}
			\begin{split}
				V_0&:=\left\{\zeta\in H^1(\Omega)\,\big|\,\int_\Omega \zeta\,\mathrm dx=0
					\right\},\\
				\tilde{V}_0&:=\left\{\zeta\in (H^1(\Omega))^*\,\big|\,
					\left\langle \zeta,\mathbf 1\right\rangle_{(H^1)^*\times H^1}=0\right\}.
			\end{split}
		\end{equation*}
		
		After passing the discretization fineness to $0$, i.e. $\tau\rightarrow 0^+$, we obtain the corresponding regularized equations
		and inequalities for \eqref{eqn:ID1}-\eqref{eqn:ID5}.
		A further passage $\delta\rightarrow 0^+$ yields a weak solution as required.
		\ep
	\end{proof}\\
	In the following, we always consider a given countable sequence $\varepsilon_k\to 0^+$ as $k\to\infty$ and omit the subscript $k$
	by simply writing $\varepsilon\to 0^+$.
	
	The existence proof is based on $\Gamma$-limit techniques of the \textit{reduced}
        energy functionals of $\C E_\varepsilon$ and $\C F_\varepsilon$ with $\C
        F_\varepsilon(c,e,z):=\int_{\Omega} W^\varepsilon(c,e,z)\dx$ in order to
	gain a suitable energy estimate in the limit $\varepsilon\rightarrow 0^+$.
        The reduced energy functionals $\frak E_\varepsilon$ and $\frak
        F_\varepsilon$ are defined by
	\begin{align*}
			\frak E_\varepsilon(c,\xi,z)&:=
			\begin{cases}
				\displaystyle\min_{\zeta\in H_D^1(\Omega;\R^n)}\C E_\varepsilon(c,\epsilon(\xi+\zeta),z)
				&\text{if }0\leq z\leq 1,\\
				\infty&\text{else,}
			\end{cases}\\
			\frak F_\varepsilon(c,\xi,z)&:=
			\begin{cases}
				\displaystyle\min_{\zeta\in H_D^1(\Omega;\R^n)}\C F_\varepsilon(c,\epsilon(\xi+\zeta),z)
				&\text{if }0\leq z\leq 1,\\
				\infty&\text{else}.
			\end{cases}
	\end{align*}
	The $\Gamma$-limits of $\frak E_\varepsilon$ and $\frak F_\varepsilon$ as $\varepsilon\rightarrow 0^+$
	exist in the topological space $H_\mathrm{w}^1(\Omega)\times W^{1,\infty}(\Omega;\R^n)\times W_\mathrm{w}^{1,p}(\Omega)$ and
	are denoted by $\frak E$ and $\frak F$, respectively.
	Here, $H_\mathrm{w}^1(\Omega)$ denotes the space $H^1(\Omega)$ with its weak topology.
	The limit functional $\frak F$ is needed as an auxiliary construction in the following because it already captures the essential properties of $\frak E$.
	In the next section, we are going to prove some properties of the $\Gamma$-limit $\frak E$ which are used in the global-in-time existence
	proof.
	
	Let $(c_\varepsilon^0,b_\varepsilon^0,z_\varepsilon^0)\rightarrow (c^0,b^0,z^0)$ as $\varepsilon\rightarrow 0^+$
	be a recovery sequence for $\frak E_\varepsilon \xrightarrow{\Gamma}\frak E$.
	In particular, $c_\varepsilon^0\weaklim c^0$ in $H^1(\Omega)$, $b_\varepsilon^0\rightarrow b^0$ in $W^{1,\infty}(\Omega;\R^n)$
	and $z_\varepsilon^0\weaklim z^0$ in $W^{1,p}(\Omega)$.
	Furthermore, we set $b_\varepsilon:=b-b^0+b_\varepsilon^0$.
	For each $\varepsilon>0$, we obtain a weak solution $(c_\varepsilon,u_\varepsilon,z_\varepsilon,\mu_\varepsilon)$ for
	$(c_\varepsilon^0,z_\varepsilon^0,b_\varepsilon)$ according to Theorem \ref{theorem:regProblem}.

        By means of the energy estimate \eqref{eqn:ID5},  
        we deduce the following a-priori estimates for our proposed model:
	\begin{lemma} 
	\label{lemma:apriori}
	There exists a constant $C>0$ independent of
	$\varepsilon$ such that
	\begin{itemize}
		\item[(i)]
			$\sup_{t\in[0,T]}\|c_\varepsilon(t)\|_{H^1(\Omega)}\leq C$,
		\item[(ii)]
			$\|\widehat e_\varepsilon\|_{L^2(\Omega_T;\mathbb R^{n\times n})}\leq C$
			with $\widehat e_\varepsilon:=e_\varepsilon\mathds{1}_{\{z_\varepsilon>0\}}$,
		\item[(iii)]
			$\sup_{t\in[0,T]}\|z_\varepsilon(t)\|_{W^{1,p}(\Omega)}\leq C$,
		\item[(iv)]
			$\|\partial_t z_\varepsilon\|_{L^2(\Omega_T)}\leq C$,
		\item[(v)]
			$\|W^\varepsilon(c_\varepsilon,e_\varepsilon,z_\varepsilon)\|_{L^\infty(0,T;L^1(\Omega))}\leq C$
			with $e_\varepsilon:=\epsilon(u_\varepsilon)$,
		\item[(vi)]
			$\|m^\varepsilon(z_\varepsilon)^{1/2}\nabla \mu_\varepsilon\|_{L^2(\Omega_T;\R^n)}\leq C$,
		\item[(vii)]
			$\|\partial_t c_\varepsilon\|_{L^2(0,T;(H^1(\Omega))^\star)}\leq \|m^\varepsilon(z_\varepsilon)\nabla\mu_\varepsilon\|_{L^2(\Omega_T;\R^n)}\leq C$.
	\end{itemize}
	\end{lemma}
	\begin{proof} 
		Applying Gronwall's lemma to the energy estimate \eqref{eqn:ID5} and noticing the boundedness of
		$\C E_\varepsilon(c_\varepsilon^0, \epsilon(u_\varepsilon^0),z_\varepsilon^0)$ with respect to
		$\varepsilon\in(0,1)$ show (iv), (vi) and
		\begin{align}
			\label{eqn:EnergyBoundedness}
			& \C E_\varepsilon(c_\varepsilon(t), e_\varepsilon(t),z_\varepsilon(t))\leq C
		\end{align}
		for a.e. $t\in(0,T)$ and all $\varepsilon\in(0,1)$ 
		and in particular (v).
		Item (vii) follows from (vi) and the diffusion equation \eqref{eqn:ID2}.
		Taking the constraint $\int_\Omega c_\varepsilon(t)\dx\equiv\;const$ and
		the restriction $0\leq z_\varepsilon\leq 1$
		into account, property \eqref{eqn:EnergyBoundedness} gives rise to
		\begin{align*}
			\|c_\varepsilon\|_{L^\infty(0,T;H^{1}(\Omega))}&\leq C,\\
			\|z_\varepsilon\|_{L^\infty(0,T;W^{1,p}(\Omega))}&\leq C.
		\end{align*}
		Together with the control of the time-derivatives (iv) and (vii), we even obtain boundedness of
		\begin{align*}
			\|c_\varepsilon(t)\|_{H^{1}(\Omega)}&\leq C,\\
			\|z_\varepsilon(t)\|_{W^{1,p}(\Omega)}&\leq C
		\end{align*}
		for every $t\in[0,T]$ and $\varepsilon\in(0,1)$. Hence, (i) and (iii) are proven.
		
		It remains to show (ii).
		By using the estimates
		\begin{align*}
			\hspace*{8em}&0\leq \chi_\varepsilon\leq 1&&\text{a.e. in }\Omega_T,\hspace*{8em}\\
			&0\leq W_{,z}^\varepsilon(c_\varepsilon,e_\varepsilon,z_\varepsilon)&&\text{a.e. in }\Omega_T,
		\end{align*}
		and the definition of $r_\varepsilon$ in \eqref{eqn:rstructure},
		we obtain (here $(\cdot)^+:=max\{\cdot, 0\}$ and $(\cdot)^-:=min\{\cdot, 0\}$)
		\begin{align}
			&\int_{\{z_\varepsilon>0\}}\big(W_{,z}^\varepsilon(c_\varepsilon,e_\varepsilon,c_\varepsilon)+f'(z_\varepsilon)\big)\dxt\notag\\
			&\qquad=\int_{\Omega_T}\big(W_{,z}^\varepsilon(c_\varepsilon,e_\varepsilon,c_\varepsilon)+f'(z_\varepsilon)\big)\dxt\notag\\
			&\qquad\quad-\int_{\{z_\varepsilon=0\}}\big(W_{,z}^\varepsilon(c_\varepsilon,e_\varepsilon,c_\varepsilon)+f'(z_\varepsilon)\big)^+\dxt
			-\int_{\{z_\varepsilon=0\}}\big(W_{,z}^\varepsilon(c_\varepsilon,e_\varepsilon,c_\varepsilon)+f'(z_\varepsilon)\big)^-\dxt\notag\\
			&\qquad\leq\int_{\Omega_T}\big(W_{,z}^\varepsilon(c_\varepsilon,e_\varepsilon,c_\varepsilon)+f'(z_\varepsilon)\big)\dxt\notag\\
			&\qquad\quad-\int_{\{z_\varepsilon=0\}}\chi_\varepsilon\big(W_{,z}^\varepsilon(c_\varepsilon,e_\varepsilon,c_\varepsilon)+f'(z_\varepsilon)\big)^+\dxt
			-\int_{\{z_\varepsilon=0\}}\big(f'(z_\varepsilon)\big)^-\dxt\notag\\
		\label{eqn:est3}
			&\qquad=\int_{\Omega_T}\big(W_{,z}^\varepsilon(c_\varepsilon,e_\varepsilon,c_\varepsilon)+f'(z_\varepsilon)+r_\varepsilon\big)\dxt
				-\int_{\{z_\varepsilon=0\}}\big(f'(z_\varepsilon)\big)^-\dxt.
		\end{align}
		To proceed, we test inequality \eqref{eqn:ID4} with $\zeta\equiv-1$ and integrate from $t=0$ to $t=T$:
	  \begin{align}
	    \label{eqn:VIestimate1}
	    \int_{\Omega_T} \big(W_{,z}^\varepsilon(c_\varepsilon, e_\varepsilon,z_\varepsilon)+f'(z_\varepsilon)
	    +r_\varepsilon \big) \dxt
	    \leq-\int_{\Omega_T}\partial_t z_\varepsilon\dxt.
	  \end{align}
	  Combining \eqref{eqn:est3} and \eqref{eqn:VIestimate1}, we end up with
		\begin{align*}
			&\int_{\{z_\varepsilon>0\}}\big(W_{,z}^\varepsilon(c_\varepsilon,e_\varepsilon,c_\varepsilon)+f'(z_\varepsilon)\big)\dxt
				\leq-\int_{\Omega_T}\partial_t z_\varepsilon\dxt-\int_{\{z_\varepsilon=0\}}\big(f'(z_\varepsilon)\big)^-\dxt.
		\end{align*}
		The right hand side is bounded because of the boundedness properties (iii) and (iv).
		But we also have ($\widehat C>0$ and $C>0$ are constants)
		\begin{align*}
			\int_{\{z_\varepsilon>0\}}\big(W_{,z}^\varepsilon(c_\varepsilon,e_\varepsilon,c_\varepsilon)+f'(z_\varepsilon)\big)\dxt
				&\geq\int_{\{z_\varepsilon>0\}}\Big(\widetilde C|e_\varepsilon|^2-C|c_\varepsilon|^2+f'(z_\varepsilon)\Big)\dxt\\
				&\geq\widetilde C\int_{\Omega_T}|\widehat e_\varepsilon|^2\dxt+\int_{\{z_\varepsilon>0\}}\Big(-C|c_\varepsilon|^2+f'(z_\varepsilon)\Big)\dxt.
		\end{align*}
		Since the left hand side is bounded and
		$\int_{\{z_\varepsilon>0\}}\Big(-C|c_\varepsilon|^2+f'(z_\varepsilon)\Big)\dxt$ is also bounded, we obtain (ii).\ep
	\end{proof}
        
	We proceed by deriving convergence properties for $c_\varepsilon$, 
        $\widehat e_\varepsilon$, $z_\varepsilon$ as $\varepsilon \to 0$ .
	\begin{lemma}
	\label{lemma:convergingSubsequences}
		There exists functions
		\begin{itemize}
			\item[]
				$c\in L^\infty(0,T;H^1(\Omega))\cap H^1(0,T;(H^1(\Omega))^\star)$ with $c(0)=c^0$,
			\item[]
				$\widehat e\in L^2(\Omega_T;\mathbb R^{n\times n})$,
			\item[]
				$z\in L^\infty(0,T;W^{1,p}(\Omega))\cap H^1(0,T;L^2(\Omega))$ with $z(0)=z^0,\quad z\leq 0,\quad\partial_t z\leq 0$
		\end{itemize}
		and a subsequence (we omit the index) such that for $\varepsilon\rightarrow 0^+$
		\begin{enumerate}
			\item[(a)] 
				$c_\varepsilon\rightharpoonup c\text{ in }H^1(0,T;(H^1(\Omega))^\star)$,\\
				$c_\varepsilon\rightarrow c\text{ in }L^\kappa(0,T;L^r(\Omega))\text{ for all }\kappa\geq 1,\;1\leq r<2^\star$,\\
				$c_\varepsilon(t)\rightharpoonup c(t)\text{ in }H^1(\Omega)$ for all $t$,\\
				$c_\varepsilon\rightarrow c\text{ a.e. in }\Omega_T$,
			\item[(b)] 
				$z_\varepsilon\rightharpoonup z\text{ in }H^1(0,T;L^2(\Omega))$,\\
				$z_\varepsilon\rightarrow z\text{ in }L^\kappa(0,T;W^{1,p}(\Omega))$ for all $\kappa\geq 1$,\\
				$z_\varepsilon(t)\weaklim z(t)\text{ in }W^{1,p}(\Omega)$ for all $t$,\\
				$z_\varepsilon\rightarrow z \text{ in }\C C(\ol{\Omega_T})$,
			\item[(c)] 
				$b_\varepsilon\rightarrow b\text{ in } W^{1,1}(0,T;W^{1,\infty}(\Omega;\R^n)),$
			\item[(d)] $\widehat e_{\varepsilon}\rightharpoonup \widehat e\text{ in }L^2(\Omega_T;\mathbb R^{n\times n})$,\\
				$W_{,e}^\varepsilon(c_\varepsilon,e_\varepsilon,z_\varepsilon)\weaklim W_{,e}(c,\widehat e,z)$
				in $L^2(\{z>0\};\mathbb R^{n\times n})$,\\
				$W_{,e}^\varepsilon(c_\varepsilon,e_\varepsilon,z_\varepsilon)\rightarrow 0$
				in $L^2(\{z=0\};\mathbb R^{n\times n})$,\\
				$W_{,c}^\varepsilon(c_\varepsilon,e_\varepsilon,z_\varepsilon)\weaklim W_{,c}(c,\widehat e,z)$
				in $L^2(\{z>0\};\mathbb R^{n\times n})$,\\
				$W_{,c}^\varepsilon(c_\varepsilon,e_\varepsilon,z_\varepsilon)\rightarrow 0$
				in $L^2(\{z=0\};\mathbb R^{n\times n})$.
		\end{enumerate}
	\end{lemma}
	\begin{proof} 
	The a-priori estimates from Lemma \ref{lemma:apriori} and classical compactness theorems as well as
	compactness theorems from Lions and Aubin \cite{Simon} yield the
  convergence properties of (a) and
  \begin{subequations}
	\begin{align}
	\label{eqn:conv1}
		\hspace*{7em}&z_\varepsilon\stackrel{\star}{\rightharpoonup} z&&\text{ in }L^\infty(0,T;W^{1,p}(\Omega)),\hspace*{7em}\\
		&z_\varepsilon\rightharpoonup z&&\text{ in }H^1(0,T;L^2(\Omega)),\\
	\label{eqn:conv3}
		&z_\varepsilon\rightarrow z&&\text{ in }\C C(\ol{\Omega_T}),\\
		&\widehat e_{\varepsilon}\rightharpoonup \widehat e&&\text{ in }L^2(\Omega_T;\mathbb R^{n\times n}),\\
		&W_{,e}^\varepsilon(c_\varepsilon, e_\varepsilon,z_\varepsilon)\weaklim w_e&&\text{ in }
		L^2(\Omega_T;\mathbb R^{n\times n})
	\end{align}
  \end{subequations}
	as $\varepsilon\rightarrow 0^+$ for a subsequence (we omit the subscript) and appropriate functions $w_e$, $\widehat e$ and $z$.
	
	The proof of the strong convergence of $\nabla z_\varepsilon$ in
	$L^p(\Omega_T;\R^n)$ essentially relies on the uniform $p$-mono\-tonicity property
	\begin{equation}
	\label{eqn:pMonotonicity}
		C_\mathrm{uc}|x-y|^p \le \big(|x|^{p-2} x - |y|^{p-2}y \big)\cdot(x-y)
	\end{equation}
	and on an approximation scheme $\{\zeta_\varepsilon\}\subseteq L^p(0,T;W^{1,p}(\Omega))$ with $\zeta_\varepsilon\geq 0$ and
	\begin{subequations}
	  \begin{align}
	    \label{eqn:approxConvergence}
	    &\zeta_\varepsilon\rightarrow z\text{ in }L^p(0,T;W^{1,p}(\Omega))\text{ as }\varepsilon\rightarrow 0^+,\\
	    \label{eqn:approxRestriction}
	    &0\leq \zeta_\varepsilon\leq z_\varepsilon\text{ a.e. in }\Omega_T\text{ for all }\varepsilon\in(0,1).
	  \end{align}
	\end{subequations}
	For details of the construction of $\zeta_\varepsilon$, we refer to \cite{WIAS1520}.
	By using \eqref{eqn:pMonotonicity}, we obtain the estimate:
	\begin{align*}
	  C_\mathrm{uc}\int_{\Omega_T}|\nabla z_\varepsilon-\nabla z|^p\dxt
	  &\leq\int_{\Omega_T}(|\nabla z_\varepsilon|^{p-2}\nabla z_\varepsilon
	  -|\nabla z|^{p-2}\nabla z)\cdot \nabla(z_\varepsilon-z)\dxt\\
	  &=\underbrace{\int_{\Omega_T}|\nabla z_\varepsilon|^{p-2}\nabla z_\varepsilon
	    \cdot \nabla(z_\varepsilon-\zeta_\varepsilon)\dxt}_{A_\varepsilon}\\
	  &\quad+\underbrace{\int_{\Omega_T}|\nabla z_\varepsilon|^{p-2}\nabla z_\varepsilon
	    \cdot \nabla(\zeta_\varepsilon-z)-|\nabla z|^{p-2}\nabla z
	    \cdot \nabla(z_\varepsilon-z)\dxt}_{B_\varepsilon}.
	\end{align*}
	The variational inequality \eqref{eqn:ID4} tested with $\zeta(t)=\zeta_\varepsilon(t)-z_\varepsilon(t)$
	and integrated from $t=0$ to $t=T$ yield
	\begin{align}
		A_\varepsilon\leq{}&
			\int_{\Omega_T}\Big(W_{,z}^\varepsilon(c_\varepsilon,\epsilon(u_\varepsilon),z_\varepsilon) + f'(z_\varepsilon)
			+\partial_t z_\varepsilon\Big)(\zeta_\varepsilon-z_\varepsilon)\dxt\notag\\
		\leq{}&
			\int_{\Omega_T}  g'(z_\varepsilon) \varphi^1\epsilon(u_\varepsilon) :   \epsilon(u_\varepsilon)
			(\zeta_\varepsilon-z_\varepsilon)\dxt
			+\int_{\Omega_T}\varphi^2(c_\varepsilon):\epsilon(u_\varepsilon)(\zeta_\varepsilon-z_\varepsilon)\dxt\notag\\
	\label{eqn:est1}
		&+\int_{\Omega_T}(\varphi^3(c_\varepsilon) + f'(z_\varepsilon)+\partial_t z_\varepsilon)(\zeta_\varepsilon-z_\varepsilon)
			\dxt.
	\end{align}
	Here, we have used the fact that $r_\varepsilon=0$ a.e. in $\{z_\varepsilon>0\}$
	and $\zeta_\varepsilon-z_\varepsilon=0$ a.e. in $\{z_\varepsilon=0\}$ by \eqref{eqn:approxRestriction}.
	We treat each term on the right hand side of \eqref{eqn:est1} as follows:
	\begin{itemize}
		\item[$\bullet$]
			Due to \eqref{eqn:assumptiong}\text{ and }\eqref{eqn:approxRestriction}, we obtain
			$$
				g'(z_\varepsilon) \varphi^1\epsilon(u_\varepsilon):\epsilon(u_\varepsilon)(\zeta_\varepsilon-z_\varepsilon)\leq 0
					\text{ a.e. in $\Omega_T$}.
			$$
			Therefore, the first integral
			on the right hand side of \eqref{eqn:est1} is less or equal $0$.
		\item[$\bullet$]
			Due to the estimates \eqref{eqn:assumptionPhi2} and \eqref{eqn:approxRestriction},
			and the estimates
			$$
				z_\varepsilon\leq C g(z_\varepsilon),
			$$
			(follows from \eqref{eqn:assumptiong} and $g(0)\geq 0$) and
			$$
				g(z_\varepsilon)|e_\varepsilon|^2\leq C\big(W^\varepsilon(c_\varepsilon,e_\varepsilon,z_\varepsilon)+|c_\varepsilon|^2\big)
			$$
			(follows from \eqref{eqn:elasticEnergy}, \eqref{eqn:assumptionPhi2} and \eqref{eqn:assumptionPhi2}),
			the second integral on the right hand side of \eqref{eqn:est1} can be estimated as follows:
			\begin{align}
				&\left|\int_{\Omega_T}\varphi^2(c_\varepsilon):\epsilon(u_\varepsilon)(\zeta_\varepsilon-z_\varepsilon)\dxt\right|\notag\\
				&\quad\leq
					\big\|\varphi^2(c_\varepsilon)\big\|_{L^\infty(0,T;L^2(\Omega;\R^{n\times n}))}
					\big\|\epsilon(u_\varepsilon)\sqrt{|\zeta_\varepsilon-z_\varepsilon|}\big\|_{L^\infty(0,T;L^2(\Omega;\R^{n\times n}))}
					\big\|\sqrt{|\zeta_\varepsilon-z_\varepsilon|}\big\|_{L^1(0,T;L^\infty(\Omega))}\notag\\
				&\quad\leq C
					\big\|c_\varepsilon\big\|_{L^\infty(0,T;L^2(\Omega;\R^{n\times n}))}
					\big\|\epsilon(u_\varepsilon)\sqrt{|z_\varepsilon|+|z_\varepsilon|}\big\|_{L^\infty(0,T;L^2(\Omega;\R^{n\times n}))}
					\big\|\sqrt{|\zeta_\varepsilon-z_\varepsilon|}\big\|_{L^1(0,T;L^\infty(\Omega))}\notag\\
				&\quad\leq C
					\big\|c_\varepsilon\big\|_{L^\infty(0,T;L^2(\Omega;\R^{n\times n}))}
					\big\||\epsilon(u_\varepsilon)|^2g(z_\varepsilon)\big\|_{L^\infty(0,T;L^1(\Omega;\R^{n\times n}))}^{1/2}
					\big\|\sqrt{|\zeta_\varepsilon-z_\varepsilon|}\big\|_{L^1(0,T;L^\infty(\Omega))}\notag\\
				&\quad\leq C
					\big\|c_\varepsilon\big\|_{L^\infty(0,T;L^2(\Omega;\R^{n\times n}))}
					\Big(\big\|W^\varepsilon(c_\varepsilon,\epsilon(u_\varepsilon),z_\varepsilon)\big\|_{L^\infty(0,T;L^1(\Omega;\R^{n\times n}))}^{1/2}
					+\big\|c_\varepsilon\big\|_{L^\infty(0,T;L^2(\Omega;\R^{n\times n}))}\Big)\times\notag\\
			\label{eqn:est2}
				&\qquad\quad\times\big\|\sqrt{|\zeta_\varepsilon-z_\varepsilon|}\big\|_{L^1(0,T;L^\infty(\Omega))}.
			\end{align}
			By using the a priori estimates in Lemma \ref{lemma:apriori} (i) and (v), we obtain
			$$
				\big\|c_\varepsilon\big\|_{L^\infty(0,T;L^2(\Omega;\R^{n\times n}))}
					\Big(\big\|W^\varepsilon(c_\varepsilon,\epsilon(u_\varepsilon),z_\varepsilon)\big\|_{L^\infty(0,T;L^1(\Omega;\R^{n\times n}))}^{1/2}
					+\big\|c_\varepsilon\big\|_{L^\infty(0,T;L^2(\Omega;\R^{n\times n}))}\Big)\leq C
			$$
			and by using the convergence properties \eqref{eqn:conv3} and \eqref{eqn:approxConvergence}
			and the compact embedding\linebreak $W^{1,p}(\Omega)\hookrightarrow\C C(\ol{\Omega_T})$, we end up with
			$$
				\big\|\sqrt{|\zeta_\varepsilon-z_\varepsilon|}\big\|_{L^1(0,T;L^\infty(\Omega))}\to 0
			$$
			as $\varepsilon\to 0^+$. Therefore, the right hand side of \eqref{eqn:est2} converges to $0$ as $\varepsilon\to 0^+$.
		\item[$\bullet$]
			Exploiting the boundedness properties in Lemma \ref{lemma:apriori} (i), (iii) and (iv)
			and the convergence properties \eqref{eqn:conv3} and \eqref{eqn:approxConvergence},
			the third integral on the right hand side of \eqref{eqn:est1} converges to $0$ as $\varepsilon\to 0^+$.
	\end{itemize}
	We have shown $\limsup_{\varepsilon\to 0^+}A_\varepsilon=0$.
	Futhermore, the boundedness of $\{\nabla z_\varepsilon\}$ in $L^p(\Omega_T;\R^n)$ (see Lemma \ref{lemma:apriori}) and the
	convergence property \eqref{eqn:approxConvergence} show $B_\varepsilon\rightarrow 0$ as $\varepsilon\rightarrow 0^+$.

	We conclude with $\lim_{\varepsilon\to 0^+}\int_{\Omega_T}|\nabla z_\varepsilon-\nabla z|^p\dxt=0$.
	Therefore, (b) is also shown.
	
	Property (c) is an immediate consequence of the definition of
	$b_\varepsilon$ and the convergence property $b^0_\varepsilon \to b^0$
	in $W^{1, \infty}(\Omega; \R^n)$.

	To prove (d), we define $N_\varepsilon:=\{z_\varepsilon>0\}\cap\{z>0\}$.
	Consequently, we get
	\begin{align}
	  \label{eqn:We}
	  &W_{,e}^\varepsilon(c_\varepsilon, \widehat e_\varepsilon,z_\varepsilon)\mathds{1}_{N_\varepsilon}
	  =W_{,e}^\varepsilon(c_\varepsilon, e_\varepsilon,z_\varepsilon)\mathds{1}_{N_\varepsilon}
	\end{align}
	and the convergence
	\begin{align}
	  \label{eqn:domainConvergence}
	  \mathds{1}_{N_\varepsilon}\rightarrow \mathds{1}_{\{z>0\}}\text{ pointwise in }\Omega_T\text{ and in }L^q(\Omega_T)\text{ for all }q\geq 1
	\end{align}
	as $\varepsilon\rightarrow 0^+$ by using $z_\varepsilon\rightarrow z$ in $\ol{\Omega_T}$.
	Calculating the weak $L^1(\Omega_T;\mathbb R^{n\times n})$-limits in
	\eqref{eqn:We} for $\varepsilon\rightarrow 0^+$ on both sides
	by using the already proven convergence properties,
	we obtain $W_{,e}(c,\widehat e,z)\mathds 1_{\{z>0\}}=w_e\mathds 1_{\{z>0\}}$.
	The remaining convergence properties in (d) follow from Lemma \ref{lemma:apriori} (v).
	\ep
	\end{proof}
      
	\begin{corollary}
	\label{cor:inclusion}
		Let $t\in[0,T]$ and $U\compact \{ z(t)=0 \}$ be an open subset.
		Then $U\subseteq \{z_\varepsilon(s)>0\}$ for all $s\in[0,t]$ provided that $\varepsilon>0$ is sufficiently small.
		More precisely, there exist $0<\eta$ such that
		\begin{align*}
			z_\varepsilon(s)\geq \eta\text{ in }U
		\end{align*}
		for all $s\in[0,t]$ and for all $0<\varepsilon\ll 1$.
	\end{corollary}
	\begin{proof}
		We have $\overline U\times[0,T]\subseteq\{z>0\}$ and $\inf_{(x,t)\in\overline{U}\times[0,t]}z(x,s)>0$ by $z\in\C C(\overline{\Omega_T})$.
		The claim follows from $z_\varepsilon\to z$ in $\C C(\overline{\Omega_T})$ according to Lemma \ref{lemma:convergingSubsequences} (b).
		\ep
	\end{proof}
	
	To obtain a-priori estimates for the chemical potentials $\{\mu_\varepsilon\}$ in a local sense, we make use of the so-called \textit{conical}
	Poincar\'e inequality for star-shaped domains cited below.
	\begin{theorem}[Conical Poincar\'e inequality \cite{BK98}]
	\label{theorem:conicalPI}
		Suppose that $\Omega\subseteq\R^n$ is a bounded and star-shaped domain, $r\geq 0$ and $1\leq p<\infty$. Then there exists a constant
		$C=C(\Omega,p,r)>0$ such that
		\begin{align*}
			\int_\Omega |w(x)-w_{\Omega,\delta^t}|^p\delta^r(x)\dx\leq C\int_\Omega|\nabla w(x)|^p\delta^r(x)\dx
		\end{align*}
		for all $w\in C^1(\Omega),$ where the $\delta^r$-weight $w_{\Omega,\delta^r}$ is given by
		\begin{align*}
			w_{\Omega,\delta^r}:=\int_\Omega w(x)\delta^r(x)\dx,\quad
			\delta(x):=\mathrm{dist}(x,\partial\Omega).
		\end{align*}
	\end{theorem}
	By a density argument, the statement is, of course, also true for all $w\in W^{1,p}(\Omega)$.
	\begin{lemma}[A-priori estimates for {${\bm \mu_\varepsilon}$}]\hspace{1em}
	\label{lemma:apriori2}
		\begin{enumerate}
			\renewcommand{\labelenumi}{(\roman{enumi})}
			\item \textit{Interior estimate.}
				For every $t\in[0,T]$ and for every open cube $Q\compact \{z(t)>0\}\cap\Omega$,
				there exists a $C>0$ such that for all $0<\varepsilon\ll1$
				\begin{align}
				\label{eqn:innerEstimate}
					\|\mu_\varepsilon\|_{L^2(0,t;H^1(Q))}&\leq C.
				\end{align}
			\item \textit{Estimate at the boundary.}
				For every $t\in[0,T]$ and every $x_0\in \{z(t)>0\}\cap\partial\Omega$,
				there exist a neighborhood $U$ of $x_0$ 
				and a $C>0$ such that for all $0<\varepsilon\ll1$
				\begin{align}
				\label{eqn:boundaryEstimate}
					\|\mu_\varepsilon\|_{L^2(0,t;H^1(U\cap\Omega))}&\leq C.
				\end{align}
		\end{enumerate}
	\end{lemma}
	\begin{proof}
		\begin{enumerate}
			\renewcommand{\labelenumi}{(\roman{enumi})}
			\item[(i)]
				Let $t\in[0,T]$ and $Q\compact \{z(t)>0\}\cap\Omega$ be an open cube.
				We consider the Lipschitz domain $\widetilde Q:=B_\eta(Q):=\{x\in\R^n\,|\,\mathrm{dist}(x,Q)<\varepsilon\}$, where $\eta>0$
				is chosen so small such that $\widetilde Q\compact \{z(t)>0\}\cap\Omega$.
				We define the following function
				\begin{align*}
					\zeta(x):=
					\begin{cases}
						\lambda(x):=\mathrm{dist}(x,\partial \widetilde Q)&\text{if }x\in \widetilde Q,\\
						0&\text{else.}
					\end{cases}
				\end{align*}
				$\zeta$ is a Lipschitz function on $\ol\Omega$ with Lipschitz constant $1$.
				By Rademacher's theorem (see, for instance, \cite[Theorem 2.2.1]{Ziemer89}), $\zeta$ is in $W^{1,\infty}(\Omega)$.
				Now, we can test \eqref{eqn:ID3} with $\zeta$.
				Then, by using the previously proven a-priori estimates, we obtain boundedness of
				\begin{align}
				\label{eqn:bd1}
					\int_{\widetilde Q} \mu_\varepsilon(x,s)\lambda(x)\dx\leq C
				\end{align}
				with respect to a.e. $s\in(0,T)$ and $\varepsilon$.
				
				Furthermore, there exists an $\eta>0$ such that $z_\varepsilon(s)\geq\eta$ in $\widetilde Q$
				for all $s\in[0,t]$ and for all $0<\varepsilon\ll 1$ (see Corollary \ref{cor:inclusion}).
				Thus, by assumption \eqref{eqn:assumptionMobility}, $m^\varepsilon(z_\delta(s))\geq\eta'>0$ holds in $\widetilde Q$ for
				all $s\in[0,t]$ and all $0<\varepsilon\ll 1$ for a common constant $\eta'>0$.
				Consequently, we get by the a-priori estimate for $m^\varepsilon(z_\varepsilon)^{1/2}\nabla \mu_\varepsilon$
				\begin{align}
				\label{eqn:bd2}
					\|\nabla\mu_\varepsilon\|_{L^2(\widetilde Q\times [0,t])}\leq C
				\end{align}
				for all $\varepsilon$.
				Applying Theorem \ref{theorem:conicalPI} (we put in $\Omega=\widetilde Q$, $r=p=2$ and $w=\mu_\varepsilon(s)$ for $s\in[0,t]$),
				integrating from $0$ to $t$ in time and using boundedness
				properties \eqref{eqn:bd1} and \eqref{eqn:bd2}, we obtain boundedness of $\|\mu_\varepsilon\lambda\|_{L^2(\widetilde Q\times [0,t])}$ and
				thus boundedness of $\|\mu_\varepsilon\|_{L^2(Q\times [0,t])}$ with respect to $0<\varepsilon\ll 1$.
				Together with \eqref{eqn:bd2}, we get the claim \eqref{eqn:innerEstimate}.
			\item[(ii)]
				Since $\{z(t)=0\}\subseteq\ol\Omega$ is a closed set, we can find a
				neighborhood $U\subseteq\R^n\setminus\{z(t)=0\}$ of $x_0$.
				Furthermore, since $\Omega$ has a $\C C^2$-boundary, there exists a $\C C^2$-diffeomorphism $\pi:(-1,1)^n\rightarrow U$ with the properties
				\begin{itemize}
					\item[\textbullet]
						$\pi\big((-1,1)^{n-1}\times(-1,0)\big)\subseteq\Omega$,
					\item[\textbullet]
						$\pi\big((-1,1)^{n-1}\times\{0\}\big)\subseteq\partial\Omega$,
					\item[\textbullet]
						$\pi\big((-1,1)^{n-1}\times(0,1)\big)\subseteq\R^n\setminus\ol\Omega$.
				\end{itemize}
				Let $\vartheta:(-1,1)^n\rightarrow(-1,1)^n$ denote the reflection $x\mapsto(x_1,\ldots,x_{n-1},-x_n)$ and
				$\C T:=\pi\circ\vartheta\circ\pi^{-1}$.
				Furthermore, let $\widetilde \mu_\varepsilon\in L^2(0,t;H^1(U))$ be defined by
				\begin{align*}
					\widetilde \mu_\varepsilon(x,s):=
					\begin{cases}
						\mu_\varepsilon(x,s)&\text{if }x\in U\cap\Omega,\\
						\mu_\varepsilon(\C T(x),s)&\text{if }x\in U\setminus\ol\Omega.
					\end{cases}
				\end{align*}
				Let $Q\compact U$ be a non-empty open cube with $x_0\in Q$.
				Then, integration by substitution with respect to the transformation $\C T$ yields
				for a.e. $s\in(0,t)$
				\begin{align}
					\int_{Q}\widetilde \mu_\varepsilon(x,s)\lambda(x)\dx
					={}&\int_{Q\cap\Omega}\mu_\varepsilon(x,s)\lambda(x)\dx\notag\\
					&+\int_{\C T(Q\setminus\Omega)}\mu_\varepsilon(x,s)\lambda(\C T(x))|\det(\nabla\C T(x))|\dx,
				\label{eqn:mueSplitting}
				\end{align}
				where the Lipschitz function $\lambda:\R^N\to\R$ is given by
				$$
					\lambda(x):=
					\begin{cases}
						\mathrm{dist}(x,\partial Q)&\text{if }x\in Q,\\
						0&\text{if }x\in \R^n\setminus Q.
					\end{cases}
				$$
				We are going to show that both terms on the right hand side of \eqref{eqn:mueSplitting} are bounded
				with respect to $\varepsilon$ and a.e. $s\in(0,t)$.
				\begin{itemize}
					\item[$\bullet$]
						Testing \eqref{eqn:ID3} with the function $\zeta=\lambda$ yields
						\begin{align*}
							\int_{Q\cap\Omega} \mu_\varepsilon(s)\lambda\dx
								={}&\int_{Q\cap\Omega}\big(\nabla c_\varepsilon(s)\cdot\nabla\lambda+\Psi_{,c}(c_\varepsilon(s))\lambda\big)\dx\\
							&+\int_{Q\cap\Omega}W_{,c}^\varepsilon(c_\varepsilon(s),\epsilon(u_\varepsilon(s)),z_\varepsilon(s))\lambda\dx.
						\end{align*}
						By the already know a-priori estimates, every integral term on the right hand side is bounded
						w.r.t. $\varepsilon$ and a.e. $s\in(0,T)$.
					\item[$\bullet$]
						The function
						\begin{align}
						\label{eqn:zetaFunction}
							\zeta(x):=
							\begin{cases}
								(\lambda(\C T(x)))|\det(\nabla\C T(x))|&\text{if }x\in\C T(Q\setminus\Omega),\\
								0&\text{if }x\in\Omega\setminus \C T(Q\setminus\Omega)
							\end{cases}
						\end{align}
						is a Lipschitz function in $\Omega$ because:
						\begin{itemize}
							\item[--]
								$\lambda\circ\C T$ is a Lipschitz function in $U\cap\Omega$
								and $\lambda\circ\C T=0$ in $\big(U\cap\Omega\big)\setminus\C T(Q\setminus\Omega)$.
								The first property follows from the Lipschitz continuity of $\lambda$ and of $\C T$
								(note that $\C T$ is a $\C C^2$-diffeomorphism).
								The latter property can be seen as follows.
								\textit{Assume the contrary.}
								Then, we find an $x\in \big(U\cap\Omega\big)\setminus\C T(Q\setminus\Omega)$
								such that $\lambda(\C T(x))>0$.
								By the definition of $\lambda$, we get $\C T(x)\in Q$.
								Since $x\in \Omega$, it follows $\C T(x)\not\in\Omega$ by the construction of $\C T$.
								Therefore, $\C T(x)\in Q\setminus\Omega$.
								This gives $x=\C T(\C T(x))\in \C T(Q\setminus\Omega)$ which is a contradiction.
							\item[--]
								$|\det(\nabla\C T)|$ is a Lipschitz function in $U\cap\Omega$
								(note that $\C T$ is a $\C C^2$-diffeomorphism).
						\end{itemize}
						Testing \eqref{eqn:ID3} with $\zeta$ from \eqref{eqn:zetaFunction} yields
						\begin{align*}
							&\int_{\C T(Q\setminus\Omega)} \mu_\varepsilon(s)(\lambda\circ\C T)|\det(\nabla\C T)|\dx\\
							&=\int_{\C T(Q\setminus\Omega)}\nabla c_\varepsilon(s)\cdot\nabla\big((\lambda\circ\C T)|\det(\nabla\C T)|\big)\dx\\
							&+\int_{\C T(Q\setminus\Omega)}\big(\Psi_{,c}(c_\varepsilon(s))
								+W_{,c}(c_\varepsilon(s),\epsilon(u_\varepsilon(s)),z_\varepsilon(s))\big)(\lambda\circ\C T)|\det(\nabla\C T)|\dx.
						\end{align*}
						By the already know a-priori estimates, every integral term on the right hand side is bounded
						w.r.t. $\varepsilon$ and a.e. $s\in(0,t)$.
				\end{itemize}
				
				For $\nabla\widetilde\mu_\varepsilon(s)$, we also get by integration via substitution:
				\begin{align}
					&\int_0^t\int_{Q}|\nabla \widetilde \mu_\varepsilon(x,s)|^2\dxs\notag\\
					&\qquad\leq\int_0^t\int_{Q\cap\Omega}|\nabla \mu_\varepsilon(x,s)|^2\dx\ds
						+\int_0^t\int_{Q\setminus\Omega}|\nabla \mu_\varepsilon(\C T(x),s)|^2|\nabla\C T(x)|^2\dx\ds\notag\\
					&\qquad=\int_0^t\int_{ Q\cap\Omega}|\nabla \mu_\varepsilon(x,s)|^2\dx\ds\notag\\
					&\qquad\quad+\int_0^t\int_{\C T(Q\setminus\Omega)}|\nabla \mu_\varepsilon(x,s)|^2|\nabla \C T(\C T(x))|^2|\det(\nabla\C T(x))|\dx\ds.
					\label{eqn:mue}
				\end{align}
				Since $Q\cap\Omega\compact \{z(t)>0\}$ and $\C T(Q\setminus\Omega)\compact \{z(t)>0\}$, we deduce
				$z_\varepsilon(s)\geq\eta$ on $Q\cap\Omega$ and on $\C T(Q\setminus\Omega)$
				for all $s\in[0,t]$ and for all sufficiently small $0<\varepsilon$ (see Corollary \ref{cor:inclusion}).
				Thus, $\nabla \mu_\varepsilon$ is bounded in $L^2((Q\cap\Omega)\times(0,t);\R^n)$ and in
				$L^2(\C T(Q\setminus\Omega)\times(0,t);\R^n)$ with respect to $0<\varepsilon\ll1$
				by also using the a-priori estimate for $m^\varepsilon(z_\varepsilon)^{1/2}\nabla\mu_\varepsilon$,
				the property $m\in \C C([0,1];\R_+)$ and assumption \eqref{eqn:assumptionMobility}.
				
				Therefore, the left hand side of \eqref{eqn:mue} is also bounded for all $0<\varepsilon\ll1$.
				The Conical Poincar\'e inequality in Theorem \ref{theorem:conicalPI} yields boundedness of
				$\widetilde\mu_\varepsilon\lambda$ in $L^2(Q\times(0,t))$.
				Finally, we can find a neighborhood $V\subseteq Q$ of $x_0$ such that $\widetilde\mu_\varepsilon$ is bounded in $L^2(0,t;H^1(V))$.
				\ep
		\end{enumerate}
	\end{proof}
	Due to the a-priori estimates for $\{\mu_\varepsilon\}$ and $\{u_\varepsilon\}$, the limit functions $\mu$ and $u$
	can only be expected to be in some space-time local Sobolev space
        $L_t^2H_{x,\mathrm{loc}}^1$ (see Section \ref{section:weakSolution}).
	In the sequel, it will be necessary to represent the maximal admissible subset of the not completely damaged area,
	i.e. $\frak A_D(\{z>0\})$, as a union of Lipschitz domains which are connected to parts of the Dirichlet boundary $D$.
	To this end, we define by $F$  the shrinking set $F:=\{z>0\}$ until the end of
        this work and obtain the following result.
	\begin{lemma}[cf. {\cite[Lemma 4.18]{WIAS1722}}]
	\label{lemma:deformationField}
		There exists a function $u\in L_t^2 H_{x,\mathrm{loc}}^1(\frak A(F);\R^n)$ such that
		$\epsilon(u)=\widehat e$ a.e. in $\frak A_\Gamma(F)$ and
		$u=b$ on the boundary $D_T\cap \frak A_\Gamma(F)$.
	\end{lemma}
	A related result can be shown for the sequence $\{\mu_\varepsilon\}$ by exploiting the estimates in Lemma \ref{lemma:apriori2}.
	To proceed, we recall a definition introduced in \cite{WIAS1722}.
	\begin{definition}[cf. {\cite[Definition 4.1]{WIAS1722}}]
	\label{def:fineCover}
		Let $H\subseteq\ol\Omega$ be a relatively open subset. We call a countable family $\{U_k\}$
		of open sets $U_k\compact H$ a fine representation for $H$ if
		for every $x\in H$ there exist an open set $U\subseteq\R^n$ with $x\in U$ and an $k\in\N$ such that
		$U\cap\Omega\subseteq U_k$.
	\end{definition}
	\begin{lemma}
	\label{lemma:chemicalPotential}
		Let a sequence $\{t_m\}\subseteq[0,T]$ containing $T$ be dense.
		There exist a fine representation $\{U_k^m\}_{k\in\N}$ for $F(t_m)$ for every $m\in\N$,
		a function $\mu\in L_t^2 H_{x,\mathrm{loc}}^1(F)$ and a
		subsequence of $\{\mu_\varepsilon\}$ (also denoted by $ \{\mu_\varepsilon\}$)
		such that for all $k,m\in\N$
		\begin{align}
			\label{eqn:muConvergence2}
			\mu_\varepsilon\weaklim\mu\text{ in } L^2(0,t_m;H^1(U_k^m))
		\end{align}
		as $\varepsilon\rightarrow 0^+$.
	\end{lemma}
	\begin{proof}
		A fine representation $\{U_k^m\}_{k\in\N}$ of $F(t_m)$ can be constructed by countably many open cubes $Q\compact F(t_m)\cap\Omega$
		together with finitely many open sets of the form $U\cap\Omega$ such that $U$ satisfies \eqref{eqn:boundaryEstimate} from
		Lemma \ref{lemma:apriori2} (ii).
		Then, we have for each $k,m\in\N$ the estimate
		$$
			\|\mu_\varepsilon\|_{L^2(0,t;H^1(U_k^m))}\leq C
		$$
		for all $0<\varepsilon\ll1$. By successively choosing sub-sequences and by a diagonal argument,
		we obtain a $\mu\in L_t^2 H_{x,\mathrm{loc}}^1(F)$ such that \eqref{eqn:muConvergence2} is satisfied
		(cf. proof of \cite[Lemma 4.18]{WIAS1722}).
		\ep
	\end{proof}

	The a-priori estimates and the convergence properties of $\{z_\varepsilon\}$
	in Lemma \ref{lemma:convergingSubsequences} and of $\{\mu_\varepsilon\}$ in Lemma \ref{lemma:chemicalPotential}, respectively, yield
	the following corollary.
	\begin{corollary}
	\label{cor:mueConvergence}
		It holds for $\varepsilon\rightarrow 0^+$:
		\begin{align*}
			&m(z_\varepsilon)\nabla\mu_\varepsilon\weaklim m(z)\nabla\mu\text{ in } L^2(F;\R^n),\\
			&m(z_\varepsilon)\nabla\mu_\varepsilon\rightarrow 0\text{ in } L^2(\Omega_T\setminus F;\R^n).
		\end{align*}
	\end{corollary}
	Now, we have all necessary convergence properties to perform the degenerate limit in \eqref{eqn:ID1}-\eqref{eqn:ID5}.
	To proceed, we need the following auxiliary result.
	\begin{lemma}
	\label{lemma:partitionOfUnity}
		Let $\{t_m\}$ and $\{U_k^m\}$ be as in Lemma \ref{lemma:chemicalPotential}.
		Then, for every compact subset $K\subseteq F$ there exist a finite set
		$I\subseteq\N$, values $m_k\in\N$, $k\in I$, and functions $\psi_k\in \C C^\infty(\ol{\Omega_T})$, $k\in I$, such that
		\begin{enumerate}
			\item[(i)]
				$K\cap\Omega_T\subseteq\bigcup_{k\in I}U_k^{m_k}\times(0,t_{m_k})$,
			\item[(ii)]
				$\mathrm{supp}(\psi_k)\subseteq \ol{U_k^{m_k}}\times[0,t_{m_k}]$,
			\item[(iii)]
				$\sum_{k\in I}\psi_k\equiv 1$ on $K$.
		\end{enumerate}
	\end{lemma}
	\begin{proof}
		We extend the family of open sets $\{V_k^m\}$ given by $V_k^m:=U_k^m\times(0,t_{m_k})$ in the following way.
		Define
		$$
			\C P:=\Big\{\{W_k^m\}_{k,m\in\N}\,\big|\,W_k^m\subseteq\R^{n+1} \text{ is open with }W_k^m\cap\Omega_T=U_k^m\times(0,t_{m_k})\Big\}.
		$$
		We see that $\C P$ is non-empty and that every totally ordered subset of $\C P$ has an upper bound with respect to the ''$\leq$`` ordering defined by
		$$
			\{W_k^m\}\leq \{\widetilde W_k^m\}\,\Leftrightarrow\,W_k^m\subseteq \widetilde W_k^m\text{ for all }k,m\in\N.
		$$
		By Zorn's lemma, we find a maximal element $\{\widetilde V_k^m\}$.
		It holds
		\begin{align}
		\label{eqn:maxElementProp}
			F\subseteq\bigcup_{k,m\in\N}\widetilde V_k^{m}.
		\end{align}
		Assume that this condition fails.
		Then, because of $F\cap \Omega_T=\bigcup_{k,m\in\N}V_k^{m}$, there exists a $p=(x,t)\in F\cap \partial(\Omega_T)$
		with $p\not\in\bigcup_{k,m\in\N}\widetilde V_k^m$.
		
		Let us consider the case $t<T$.
		Since $F\subseteq\ol{\Omega_T}$ is relatively open, we find an $m_0\in\N$ with $x\in F(t_{m_0})$ and $t_{m_0}>t$.
		By the fine representation property of $\{U_k^{m_0}\}_{k\in\N}$ for $F(t_{m_0})$,
		we find an open set $U\subseteq\R^n$ with $x\in U$ and $k_0\in\N$ such that $U\cap\Omega\subseteq U_{k_0}^{m_0}$.
		
		The family $\{\widetilde W_k^m\}$ given by
		$$
			\widetilde W_k^m:=
			\begin{cases}
				\widetilde V_k^m\cup U\times(-\infty,t_{m_0})&\text{if }k=k_0\text{ and }m=m_0,\\
				\widetilde V_k^m&\text{else,}
			\end{cases}
		$$
		satisfies $\{\widetilde W_k^m\}\in\C P$ and $p\in\bigcup_{k,m\in\N}\widetilde W_k^m$ which contradicts the maximality
		property of $\{\widetilde V_k^{m}\}$.
		
		In the case $t=T$, we also find $k_0,m_0\in\N$ and an open set $U\subseteq\R^n$ with $x\in U$ such that $U\cap\Omega\subseteq U_{k_0}^{m_0}$
		and $t_{m_0}=T$.
		The family $\{\widetilde W_k^m\}$ given by
		$$
			\widetilde W_k^m:=
			\begin{cases}
				\widetilde V_k^m\cup U\times\R&\text{if }k=k_0\text{ and }m=m_0,\\
				\widetilde V_k^m&\text{else,}
			\end{cases}
		$$
		also contradicts the maximality of $\{\widetilde V_k^m\}$.
		Therefore, \eqref{eqn:maxElementProp} is proven.
		
		Heine-Borel theorem yields
		$$
			K\subseteq \bigcup_{k\in I}\widetilde V_k^{m_k}
		$$
		for a finite set $I\subseteq\N$ and values $m_k\in\N$, $k\in I$. Together with a partition of unity argument, we
		get functions $\psi_k\in\C C^\infty(\ol{\Omega_T})$ such that (i)-(iii) hold.
		\ep
	\end{proof}\\
	The degenerate limit $\varepsilon\rightarrow0^+$ can be performed as follows:
	\begin{itemize}
		\item[\textbullet]
			We define the strain by $e:=\widehat e|_F\in L^2(F;\R^{n\times n})$ and obtain for the remaining variables
			\begin{align*}
			\begin{aligned}
				&c\in L^\infty(0,T;H^1(\Omega))\cap H^1(0,T;(H^1(\Omega))^\star),
				&&u\in L_t^2 H_{x,\mathrm{loc}}^1(\frak A_D(F);\R^n),\\
				&z\in L^\infty(0,T;W^{1,p}(\Omega))\cap H^1(0,T;L^2(\Omega)),
				&&\mu\in L_t^2 H_{x,\mathrm{loc}}^1(F)
			\end{aligned}
			\end{align*}
			with $e=\epsilon(u)$ in $\frak A(F)$.
		\item[\textbullet]
			Passing to the limit $\varepsilon\rightarrow0^+$ in \eqref{eqn:ID1}, \eqref{eqn:ID4} and \eqref{eqn:ID5} imply
			properties \eqref{eqn:forceBalanceWeak},
                        \eqref{eqn:VI} and \eqref{eqn:EI}, cf. \cite{WIAS1722}.
		\item[\textbullet]
			Using Lemma \ref{lemma:convergingSubsequences} (a) and Corollary \ref{cor:mueConvergence},
			we can pass to $\varepsilon\rightarrow 0^+$ in \eqref{eqn:ID2} and obtain \eqref{eqn:diffusion}.
		\item[\textbullet]
			Let $\zeta\in L^2(0,T;H^1(\Omega))$ with $\mathrm{supp}(\zeta)\subseteq F$ be a test-function.
			Furthermore, let $\{\psi_l\}$ be a partition of unity of the compact set $K:=\mathrm{supp}(\zeta)$
			according to Lemma \ref{lemma:partitionOfUnity}.
			For each $l\in\N$, we obtain $\mathrm{supp}(\zeta\psi_l)\subseteq \ol{U_l^{m_l}}\times[0,t_{m_l}]$.
			Then, integrating \eqref{eqn:ID3} in time from $0$ to $t_{m_l}$, testing the result with
			$\zeta\psi_l$ and passing to $\varepsilon\rightarrow 0^+$ by using Lemma \ref{lemma:convergingSubsequences}
			and Lemma \ref{lemma:chemicalPotential} show
			\begin{align*}
				\int_0^{t_m}\int_\Omega\mu\zeta\psi_l\dxs
				=\int_0^{t_m}\int_\Omega\Big(\nabla c\cdot\nabla(\zeta\psi_l)
					+\Psi_{,c}(c)\zeta\psi_l+W_{,c}(c,\widehat e,z)\zeta\psi_l\Big)\dxs.
			\end{align*}
			Summing with respect to $l\in I$ and noticing $\sum_{l\in I}\psi_l\equiv 1$ on $\mathrm{supp}(\zeta)$ yield 
			\eqref{eqn:chemicalPotential}.
	\end{itemize}
	In conclusion, the limit procedure in this section yields functions $(c,e,u,z,\mu)$ 
	with $e=\epsilon(u)$ in $\frak A_D(F)$ which satisfy properties (ii)-(v) of Definition \ref{def:weakSolution}.
	In particular, the damage function $z$ has no jumps with respect to time.
	However, we cannot ensure that $\{z>0\}$ equals $\frak A_D(\{z>0\})$ and, moreover, if $F\setminus\frak A_D(\{z>0\})\neq\emptyset$, it
	is not clear whether $u$ can be extended to a function on $F$ such that $e=\epsilon(u)$ also holds in $F$.
	This issue is addressed in the next section where such limit functions are concatenated in order to obtain 
	global-in-time weak solutions with fineness $\eta >0$ by Zorn's lemma.

\section{Existence results}
\label{section:existence}
	In this section, we are going to prove the main results of this
        paper. 
	
	\begin{theorem}[Maximal local-in-time weak solutions]
	\label{theorem:localInTimeExistence}
		\hspace{0.01em}\\
		Let $b\in W^{1,1}(0,T;W^{1,\infty}(\Omega;\R^n))$, $c^0\in H^1(\Omega)$ and
		$z^0\in W^{1,p}(\Omega)$ with $0<\kappa\leq z^0\leq 1$ in $\Omega$ be
		initial-boundary data.
		Then there exist a maximal value $\widehat T>0$ with $\widehat T\leq T$ and functions $c$, $u$, $z$, $\mu$ defined on the time interval $[0,\widehat T]$
		such that $(c,u,z,\mu)$ is a weak solution according to Definition \ref{def:weakSolution}.
		Therefore, if $\widehat T<T$, $(c,u,z,\mu)$ cannot be extended to a weak solution on $[0,\widehat T+\varepsilon]$.
	\end{theorem}
	\begin{proof}
          Zorn's lemma can be applied to the set
		\begin{align*}
			\C P:=\big\{(\widehat T,c,u,z,\mu)\,|\,&0<\widehat T\leq T\text{ and }(c,u,z,\mu)
				\text{ is a weak solution on}\\
			&\text{$[0,\widehat T]$ according to Definition \ref{def:weakSolution}}\big\}
		\end{align*}
		to find a maximal element with respect to the following partial ordering
		\begin{align}
			(\widehat T_1,c_1,u_1,z_1,\mu_1)\leq (\widehat T_2,c_2,u_2,z_2,\mu_2)\quad\Leftrightarrow\quad&
				\widehat T_1\leq \widehat T_2,\,c_2|_{[0,\widehat T_1]}=c_1,\,u_2|_{[0,\widehat T_1]}=u_1,\notag\\
			&z_2|_{[0,\widehat T_1]}=z_1,\,\mu_2|_{[0,\widehat T_1]}=\mu_1.
		\label{eqn:ordering}
		\end{align}
		Indeed, $\C P\neq\emptyset$ by the result in Section \ref{section:degLimit}.
		More precisely, since $z\in L^\infty(0,T;W^{1,p}(\Omega))\cap H^1(0,T;L^2(\Omega))$ and since $0<\kappa\leq z^0$, we find an $\varepsilon>0$
		such that $\{z(t)>0\}=\frak A_D(\{z(t)>0\})$ for all $t\in[0,\varepsilon]$.
		For the proof that every totally ordered subset of $\C P$ has an upper bound,
		we refer to \cite{WIAS1722}.
		\ep
	\end{proof}

	The proof of global-in-time existence of weak solutions with fineness
  $\eta$ requires	a concatenation property (see Lemma \ref{lemma:concatenation}) which is, in turn, based on some deeper insights into the
	$\Gamma$-limit $\frak E$ introduced in Section \ref{section:degLimit}.
	To this end, it is necessary to have more information about the recovery sequences for $\frak F_\varepsilon\xrightarrow{\Gamma}\frak F$.
	
	We will introduce the following substitution method:
	Assume that $u\in H^1(\Omega;\R^n)$ minimizes $\C F_\varepsilon(c,\epsilon(\cdot),z)$ with Dirichlet data $\xi$ on $D$.
	Then, by expressing the elastic energy density $W$ in terms of its derivative $W_{,e}$, i.e.
	$$
		W(c,e,z)=\frac 12 W_{,e}(c,e,z):e+\frac 12z\varphi^2(c):e+z\varphi^3(c),
	$$
	and by testing the momentum balance equation with $\zeta=u-\widetilde u$
	for a function $\widetilde u\in H^1(\Omega;\R^n)$ with $\widetilde u=\xi$ on $D$, the elastic energy term in
	$\C F_\varepsilon$ can be rewritten as
	\begin{align}
	\label{eqn:weakEnergyTransform}
		\int_\Omega W_\varepsilon(c,\epsilon(u),z)\dx
		=\int_\Omega (g(z)+\varepsilon)\left(\varphi^1 \epsilon(u):\epsilon(\widetilde u)+\frac 12 \varphi^2(c):(\epsilon(u)
			+\epsilon(\widetilde u))+\varphi^3(c)\right)\dx.
	\end{align}
	For convenience, in the following proof, we define the density function $\widetilde W$ as
	$$
		\widetilde W_\varepsilon(c,e,e_1,z):=(g(z)+\varepsilon)\left(\varphi^1 e:e_1+\frac 12 \varphi^2(c):(e+e_1)+\varphi^3(c)\right).
	$$
	\begin{lemma}
	\label{lemma:recoverySequence}
		For every $c\in H^1(\Omega)$, $\xi\in W^{1,\infty}(\Omega)$ and $z\in W^{1,p}(\Omega)$
		there exists a sequence $\delta_\varepsilon\rightarrow 0^+$ such that
		$(c,\xi,(z-\delta_\varepsilon)^+)\rightarrow (c,\xi,z)$ is a recovery sequence for $\frak F_\varepsilon\xrightarrow{\Gamma}\frak F$.
	\end{lemma}
	\begin{proof}
		We follow the idea of the proof in \cite[Lemma 4.9]{WIAS1722}.
		But here we have to deal with the additional concentration variable which complicates the calculation.
		Let $(c_\varepsilon,\xi_\varepsilon,z_\varepsilon)\rightarrow (c,\xi,z)$ be a recovery sequence and $\delta_{\varepsilon}\to0^+$
		such that $(z-\delta_\varepsilon)^+\leq z_\varepsilon$.
		Consider
		\begin{align*}
			\frak F_\varepsilon(c,\xi,(z-\delta_\varepsilon)^+)-\frak F_\varepsilon(c_\varepsilon,\xi_\varepsilon,z_\varepsilon)
			=\underbrace{\frak F_\varepsilon(c,\xi,(z-\delta_\varepsilon)^+)-\frak F_\varepsilon(c,\xi,z_\varepsilon)}_{A_\varepsilon}
			+\underbrace{\frak F_\varepsilon(c,\xi,z_\varepsilon)-\frak F_\varepsilon(c_\varepsilon,\xi_\varepsilon,z_\varepsilon)}_{B_\varepsilon}.
		\end{align*}
		We have $A_\varepsilon\leq 0$ since $\C F_\varepsilon(c,\e(\xi+\zeta),(z-\delta_\varepsilon)^+)\leq\C F_\varepsilon(c,\e(\xi+\zeta),z_\varepsilon)$
		for all $\zeta\in H_D^1(\Omega;\R^n)$.
		Now, we focus on the second term of the right hand side.
		Let $u_\varepsilon,v_\varepsilon\in H_D^1(\Omega;\R^n)$ be given by
		\begin{align*}
			u_\varepsilon=\mathop\mathrm{arg\,min}_{\zeta\in H_D^1(\Omega;\R^n)}\C F_\varepsilon(c,\epsilon(\xi+\zeta),z_\varepsilon),\quad
			v_\varepsilon=\mathop\mathrm{arg\,min}_{\zeta\in H_D^1(\Omega;\R^n)}
				\C F_\varepsilon(c_\varepsilon,\epsilon(\xi_\varepsilon+\zeta),z_\varepsilon).
		\end{align*}
		By using the substitution \eqref{eqn:weakEnergyTransform} for $W_\varepsilon(c,\e(\xi+u_\varepsilon),z_\varepsilon)$
		with test-function $\widetilde u=v_\varepsilon$ and \eqref{eqn:weakEnergyTransform} for
		$W_\varepsilon(c_\varepsilon,\e(\xi_\varepsilon+v_\varepsilon),z_\varepsilon)$ with test-function $\widetilde u=u_\varepsilon$,
		we obtain a calculation as follows:
		\begin{align*}
			B_\varepsilon={}&\C F_\varepsilon(c,\epsilon(\xi+u_\varepsilon),z_\varepsilon)
				-\C F_\varepsilon(c_\varepsilon,\epsilon(\xi_\varepsilon+v_\varepsilon),z_\varepsilon,)\\
			={}&\int_{\Omega}\Big(\widetilde W(c,\epsilon(\xi+u_\varepsilon),\epsilon(\xi+v_\varepsilon),z_\varepsilon+\varepsilon)
				-\widetilde W(c_\varepsilon,\epsilon(\xi_\varepsilon+v_\varepsilon),\epsilon(\xi_\varepsilon+u_\varepsilon),z_\varepsilon+\varepsilon)\Big)\dx\\
			={}&\int_\Omega(g(z_\varepsilon)+\varepsilon)\Big(
				\varphi^1 \epsilon(\xi+u_\varepsilon):\epsilon(\xi+v_\varepsilon)
				-\varphi^1 \epsilon(\xi_\varepsilon+u_\varepsilon):\epsilon(\xi_\varepsilon+v_\varepsilon)\\
			&\qquad
				+\frac 12\varphi^2(c):\epsilon(2\xi+u_\varepsilon+v_\varepsilon)
				-\frac 12\varphi^2(c_\varepsilon):\epsilon(2\xi+v_\varepsilon+u_\varepsilon)
				+\varphi^3(c)-\varphi^3(c_\varepsilon)\Big)\dx\\
			={}&\int_\Omega(g(z_\varepsilon)+\varepsilon)\Big(\varphi^1\epsilon(\xi):\epsilon(\xi)-\varphi^1\epsilon(\xi_\varepsilon):\epsilon(\xi_\varepsilon)
				+\varphi^1\epsilon(u_\varepsilon+v_\varepsilon):\epsilon(\xi-\xi_\varepsilon)\\
			&\qquad
				+\varphi^2(c):\epsilon(\xi-\xi_\varepsilon)+\frac12(\varphi^2(c)-\varphi^2(c_\varepsilon)):\epsilon(2\xi_\varepsilon+u_\varepsilon+v_\varepsilon)
				+\varphi^3(c)-\varphi^3(c_\varepsilon)\Big)\dx\\
			\leq{}&\int_\Omega(g(z_\varepsilon)+\varepsilon)\Big(\varphi^1\epsilon(\xi):\epsilon(\xi)-\varphi^1\epsilon(\xi_\varepsilon):\epsilon(\xi_\varepsilon)
				+\varphi^2(c):\epsilon(\xi-\xi_\varepsilon)+\varphi^3(c)-\varphi^3(c_\varepsilon)\Big)\dx\\
			&+\|(g(z_\varepsilon)+\varepsilon)\varphi^1\epsilon(u_\varepsilon+v_\varepsilon)\|_{L^2(\Omega)}\|\epsilon(\xi-\xi_\varepsilon)\|_{L^2(\Omega)}\\
			&+\frac12\|\varphi^2(c)-\varphi^2(c_\varepsilon)\|_{L^2(\Omega)}\Big(\|(g(z_\varepsilon)+\varepsilon)\epsilon(\xi_\varepsilon+u_\varepsilon)\|_{L^2(\Omega)}
				+\|(g(z_\varepsilon)+\varepsilon)\epsilon(\xi_\varepsilon+v_\varepsilon)\|_{L^2(\Omega)}\Big)
		\end{align*}
		
		Using the convergence properties $c_\varepsilon\weaklim c$ in $H^1(\Omega)$, $\xi_\varepsilon\rightarrow \xi$ in $W^{1,\infty}(\Omega)$,
		$z_\varepsilon\weaklim z$ in $W^{1,p}(\Omega)$ and the boundedness of $\C F_\varepsilon(c,\epsilon(\xi+u_\varepsilon),z_\varepsilon)$ and 
		$\C F_\varepsilon(c_\varepsilon,\epsilon(\xi_\varepsilon+v_\varepsilon),z_\varepsilon)$ with respect to $\varepsilon$,
		we conclude $\limsup_{\varepsilon\rightarrow0^+}B_\varepsilon\leq 0$.
		Now we can proceed as in the proof of \cite[Lemma
                  4.9]{WIAS1722} and the claim follows.\ep
	\end{proof}
	\begin{remark}
	\label{remark:gammaLimit}
		The knowledge of such recovery sequences for $\frak F_\varepsilon\xrightarrow{\Gamma}\frak F$ gives also more information about $\frak E$.
		In particular, we obtain that the result in Lemma \ref{lemma:recoverySequence} with $\frak E_\varepsilon\xrightarrow{\Gamma}\frak E$ instead of
		$\frak F_\varepsilon\xrightarrow{\Gamma}\frak F$ also holds true and, moreover,
		that the following properties are satisfied (cf. \cite[Corollary 4.10, Lemma 4.11]{WIAS1722}):\\
		
		\begin{tabular}{ll}
			\begin{minipage}{12.7em}
				\begin{itemize}
					\item[\textbullet]
						$\frak E(c,\xi,\mathds 1_F z)\leq\frak E(c,\xi,z)$\\
					\item[\textbullet]
						$\frak E(c,\xi,z)\leq \C E(c,\epsilon(u),z)$\\
				\end{itemize}
			\end{minipage}
			&
			\begin{minipage}{29em}
				\begin{itemize}
					\item[]
						$\forall c\in H^1(\Omega),\;\forall \xi\in W^{1,\infty}(\Omega;\R^n),\;\forall z\in W^{1,p}(\Omega)$\\
						$\forall F\subseteq\Omega\text{ open with }\mathds 1_F z\in W^{1,p}(\Omega)$,
					\item[]
						$\forall c\in H^1(\Omega),\;\forall \xi\in W^{1,\infty}(\Omega;\R^n),\;\forall z\in W^{1,p}(\Omega)\text{ with }\\0\leq z\leq 1,$
						$\forall u\in H_\mathrm{loc}^1(\{z>0\};\R^n)\text{ with }u=\xi\text{ on }D\cap\{z>0\}$.
				\end{itemize}
			\end{minipage}
		\end{tabular}
	\end{remark}
	\begin{lemma}
	\label{lemma:concatenation}
		Let $t_1<t_2<t_3$ be real numbers and let $\eta>0$.
		Suppose that
		\begin{align*}
			&\text{$\widetilde q:=(\widetilde c,\widetilde
                    e,\widetilde u,\widetilde z,\widetilde\mu,\widetilde F)$
                    is a weak solution with fineness $\eta$ on $[t_1,t_2]$},\\
			&\text{$\widehat q:=(\widehat c,\widehat e,\widehat
                    u,\widehat z,\widehat\mu,\widehat F)$ 
                    is a weak solution with fineness $\eta$ on $[t_2,t_3]$
										 with $\widehat{\frak e}_{t_2}^+=\frak E(\widehat c(t_2),\widehat b(t_2),\widehat z^+(t_2))$}\\
			&\qquad\;\text{(the value $\frak e_{t_2}^+$ for the weak solution $\widehat q$ in Definition \ref{def:weakSolution}).}
		\end{align*}
		Furthermore, suppose
		the compatibility condition
		$\widehat c(t_2)=\widetilde c(t_2)$ and
		$\widehat z^+(t_2)=\widetilde z^-(t_2)\mathds 1_{\frak A_D(\{\widetilde z^-(t_2)>0\})}$
		and the Dirichlet boundary data $b\in W^{1,1}(t_1,t_3;W^{1,\infty}(\Omega;\R^n))$.
		
		Then, we obtain that $q:=(c,e,u,z,\mu,F)$ defined as
		$q|_{[t_1,t_2)}:=\widetilde q$ and
		$q|_{[t_2,t_3]}:=\widehat q$
		is a weak solution with fineness $\eta$ on $[t_1,t_3]$.
	\end{lemma}
	\begin{proof}
		Because of the properties in Remark \ref{remark:gammaLimit} we
                can prove
		the following crucial energy estimate at time point $t_2$:
		\begin{align*}
			\lim_{s\rightarrow t_2^-}\mathop\mathrm{ess\,inf}_{\tau\in(s,t_2)}\C E(c(\tau),e(\tau),z(\tau))
			&=\lim_{s\rightarrow t_2^-}\mathop\mathrm{ess\,inf}_{\tau\in(s,t_2)}\C E(c(\tau),e(\tau),z^-(\tau))\\
			&\geq\lim_{s\rightarrow t_2^-}\mathop\mathrm{ess\,inf}_{\tau\in(s,t_2)}\C E(c(\tau),\epsilon(u(\tau)),
				z^-(\tau)\mathds 1_{\frak A_D(\{ z^-(\tau)>0\})})\\
			&\geq \lim_{s\rightarrow t_2^-}\mathop\mathrm{ess\,inf}_{\tau\in(s,t_2)}\frak E(c(\tau),b(\tau),
				z^-(\tau)\mathds 1_{\frak A_D(\{ z^-(\tau)>0\})})\\
			&\geq \frak E (c(t_2),b(t_2),\chi)\\
			&\geq \frak E (c(t_2),b(t_2),z^+(t_2))
		\end{align*}
		with $\chi:=z^-(t_2)\mathds 1_{\bigcap_{\tau\in(t_1,t_2)}\frak A_D(\{ z^-(\tau)>0\})}$.
		With this estimate, we can verify the claim by the same
                argumentation as for \cite[Lemma 4.21]{WIAS1722}.
		\ep
	\end{proof}

	\begin{theorem}[Global-in-time weak solutions with fineness $\eta$]
	\label{theorem:globalInTimeExistence}
		\hspace{0.01em}\\
		Let $b\in W^{1,1}(0,T;W^{1,\infty}(\Omega;\R^n))$, $c^0\in H^1(\Omega)$ and
		$z^0\in W^{1,p}(\Omega)$ with $0\leq z^0\leq 1$ in $\Omega$ and $\{z^0>0\}$ admissible with respect to $D$ be
		initial-boundary data.
		Furthermore, let $\eta>0$.
		Then there exists a weak solution $(c,e,u,z,\mu)$ with fineness $\eta>0$ according
		to Definition \ref{def:approxWeakSolution}.
	\end{theorem}
	\begin{proof}
		This result can also be proven by using Zorn's lemma on the set
		\begin{align*}
			\C P:=\big\{(\widehat T,c,e,u,z,\mu, F)\,|\,&0<\widehat T\leq T\text{ and }(c,e,u,z,\mu,F)
				\text{ is a weak solution with fineness $\eta$
                                  on}\\
			&\text{$[0,\widehat T]$ according to Definition \ref{def:approxWeakSolution}}\big\}
		\end{align*}
		with an ordering analogously to \eqref{eqn:ordering}.
		The assumptions for Zorn's lemma can be proven as in Theorem \ref{theorem:localInTimeExistence}
		(see \cite[Proof of Theorem 4.1]{WIAS1722}).
		To show that a maximal element from $\C P$ is actually a weak
                solution with fineness $\eta$ on the time-interval $[0,T]$,
		we need the concatenation property in Lemma \ref{lemma:concatenation}.
		Indeed, if a maximal element $\widetilde q$ is only defined on a time-interval $[0,\widetilde T]$ with $\widetilde T<T$
		we can apply the degenerated limit procedure in Section \ref{section:degLimit} to the initial values $c(\widetilde T)$
		and $z(\widetilde T)$ to obtain a new limit function $\widehat q$.
		By exploiting Lemma \ref{lemma:concatenation}, $q$ is a weak
                solution with fineness $\eta$ on the time-interval
		$[0,\widetilde T+\varepsilon]$ for a small $\varepsilon>0$ which contradicts the maximality of $\widetilde q$.
		\ep
	\end{proof}

	\addcontentsline{toc}{chapter}{Bibliography}{\footnotesize{\setlength{\baselineskip}{0.2 \baselineskip}
	\bibliography{references}
	}
	\bibliographystyle{alpha}}
	
\end{document}